\documentclass[oneside,english]{amsart}
\usepackage{geometry}
\usepackage{setspace}

\usepackage{peambl}

\newcommand{\Aviv}[1]{{\color{purple}{Aviv: #1}}}

\title{Compact Median Algebras are $\mu$-Boundaries in a unique way}

\author{Uri Bader}

\author{Aviv Taller}
\date{}

\begin{document}

\maketitle
\begin{abstract}
We consider group actions on compact median algebras.
We show that, given a generating probability measure $\mu$ on the acting group and under suitable conditions on the median algebra, it could be realized in a unique way as a $\mu$-boundary in the sense of Furstenberg. 
Along the way, we prove some structural results.
\end{abstract}
\section{Introduction}

In this work we study group actions on compact topological median algebras.
Our main goal is to understand \textit{stationary measures} for such actions.
This is a continuation of \cite{BaderTaller}, where we assumed the acting group is amenable and studied the associated invariant measures.
However, here we remove the amenability assumption, thus invariant measures do not exist in general, while stationary measures do exist. 
Our main result is Theorem~\ref{thm:stationary-nofactors} which shows that, under suitable assumptions, there exists a unique stationary measure making the given median space a boundary in the sense of Furstenberg. 

We first state Proposition~\ref{prop:structure} below, which is a general structural result - it describes a \textit{maximal cubical factor} of a median algebra.
Recall that a \emph{cube} is a median algebra isomorphic to $\{0,1\}^n$ for some cardinal $n$, called the \emph{dimension} of the cube. A finite dimensional cube is called a \emph{finite cube}. 

Given a topological group $G$, a topological median algebra endowed with an action of $G$ by automorphisms such that the action map $G\times M\to M$ is continuous is called a topological $G$-median algebra.
A $G$-equivariant continuous median map of topological $G$-median algebras is called a \emph{morphism}. 
As usual in the theory of dynamical systems, we study \emph{minimal actions}, that is we assume that our median algebra does not contain any non-empty, proper closed invariant sub-median algebra.
As in \cite{BaderTaller}, we will always assume the underlying median algebra is \textit{sclocc}, that is \textit{second countable, locally open-convex and compact}, see \S\ref{sec:medianalgebras} for precise definitions.
Moreover, we impose one more condition on our algebras - we assume that all intervals are countable (which includes finite, by our convention).

\begin{mpro} \label{prop:structure}
Let $G$ be a topological group and let $M$ be a $G$-median algebra.
Assume that $M$ is sclocc and with countable intervals.
Then there exist topological $G$-median algebras $M'$ and $C$ such that $M\simeq M' \times C$, $C$ is a finite cube and every surjective morphism $M\to C'$, where $C'$ is a $G$-cube, is the composition of the projection map $M\to C$ and a unique morphism $C\to C'$.
\end{mpro}

In case $M\simeq M'$, that is $C$ is trivial in Proposition~\ref{prop:structure}, we say that $M$ \textit{has no cubical factor}.
Recall that, given a probability measure $\mu$ on $G$, a probability measure $\nu$ on $M$ is said to be \textit{$\mu$-stationary} if $\mu*\nu=\nu$, where $\mu*\nu$ is the push forward of the measure $\mu\times \nu$ under the action map $G\times M\to M$.

\setcounter{mthm}{1}
\begin{mthm}\label{thm:stationary-nofactors}
Let $G$ be a countable discrete group and let $M$ be a minimal topological $G$-median algebra.
Assume that $M$ is sclocc with countable intervals and with no cubical factor.
Let $\mu$ be a generating probability measure on $G$.
Then there exists a unique $\mu$-stationary probability measure $\nu$ on $M$.
Moreover, the $G$-space $(M,\nu)$ is a $\mu$-boundary in the sense of Furstenberg.
\end{mthm}

Given $\mu$, a generating probability measure on $G$, we recall Furstenberg's correspondence between $G$-maps $B(G,\mu)\to \Prob(M)$ and $\mu$-stationary measures on $M$, where $B(G,\mu)$ denotes the corresponding \textit{Furstenberg-Poisson boundary} of $G$, see \cite[Theorem 2.16]{Bader2006}.
The uniqueness of the stationary measure $\nu$ in Theorem~\ref{thm:stationary-nofactors} is thus equivalent to having a unique $G$-map $\phi:B(G,\mu)\to \Prob(M)$ and the fact $(M,\nu)$ is a $\mu$-boundary is equivalent to fact that the image of $\phi$ consists of $\delta$-measures in $\Prob(M)$.

We will recall in \S\ref{sec:ergodic} the definition of a \textit{boundary pair} for $G$, as discussed in \cite{Bader2014}.
By \cite[Theorem 2.7]{Bader2014}, the future and past Furstenberg-Poisson boundaries associated with a generating probability measure form a boundary pair.
Theorem~\ref{thm:stationary-nofactors} will follow from the following.

\begin{mthm}\label{thm:nocubicalfactor}
Let $G$ be a countable discrete group and let $M$ be a minimal topological $G$-median algebra.
Assume that $M$ is sclocc with countable intervals and with no cubical factor.
Let $(B_-,B_+)$ be a boundary pair for $G$.
Then there are unique (up to a.e equivalence) a.e defined measurable $G$-maps $\phi_-:B_-\to M$ and $\phi_+:B_+\to M$. Moreover, these maps are the unique (up to a.e equivalence) a.e defined measurable $G$-maps $B_-\to \Prob(M)$, $B_+\to \Prob(M)$.
\end{mthm}

We denote by $\Cub(M)$ the set of all cubes in $M$ and endow it with the Chabouty topology.
Any measurable map $B_{\pm}\rightarrow \Cub(M)$, give rise to a map $B_{\pm}\rightarrow \Prob(M)$, by sending every cube to its unique uniform probability measure.
Thus, theorem \ref{thm:nocubicalfactor} implies in particular the uniqueness of the maps $B_{\pm}\rightarrow \Cub(M)$.
The following is the generalization of the above theorem, where we consider the possibility of a non-trivial cubical factor.

\begin{mthm} \label{thm:main}
Let $G$ be a countable discrete group and let $M$ be a minimal topological $G$-median algebra.
Assume that $M$ is sclocc with countable intervals.
Let $(B_-,B_+)$ be a boundary pair for $G$.
Then, there exist unique (up to a.e equivalence) a.e defined measurable $G$-maps $\phi_-:B_-\to \Cub(M)$ and $\phi_+:B_+\to \Cub(M)$.
Under the decomposition $M\simeq M' \times C$ given in Proposition~\ref{prop:structure},
$\phi_-$ and $\phi_+$ correspond to the products of the unique maps $\phi'_-:B_-\to M'$ and $\phi'_+:B_+\to M'$ given in Theorem~\ref{thm:nocubicalfactor} and the constant maps 
$B_\pm\to \{C\} \in \Cub(C)$.
\end{mthm}

\begin{example}
    Let $G:=F_2\times \{\pm1\}$ where $F_2:=\langle a, b\rangle$ is the free group generated by two elements, $\{\pm1\}\cong \bbZ/2$.
    Let $M:= (T_4\cup \partial T_4)\times \{\pm1\}$, where the 4-regular tree $T_4$ is identified with the Cayley graph of $F_2$ and $\partial T_4$ with its boundary.
    We take the usual compact median structure on $T_4\cup \partial T_4$ and the corresponding product structure on $M$.
    We consider the natural action of $G$ on $M$.  
This action is clearly minimal.
With respect to the measure $\mu(\{(x,\pm1)\})=1/8$ for $x\in \{a^{\pm 1},b^{\pm 1}\}$, the Furstenberg Poisson boundary $B(G,\mu)$ of $G$ is isomorphic to $\partial T_4$ endowed with the standard stationary measure. 
Then the map 
\[ \partial T_4 \to \Cub(M), \quad  \theta \mapsto \theta \times \{\pm1\} \]
is the unique 
(up to a.e equivalence) a.e defined measurable $G$-map $\partial T_4 \to \Cub(M)$.
\end{example}

It is worth noting that in case of a non-trivial cubical factor, it is no longer true that there are unique maps $B_{\pm}\to \Prob(M)$.
The following example illustrates this fact.

\begin{example}
    Consider the subgroup $G\leq \{\pm1\}^3$ consisting of all the elements $g=(x,y,z)$ such that $xyz=1$.
    Note that its action on the cube $\{\pm1\}^3$ is minimal.
    For every $t\in [0,1]$, let $\mu_t$ be the probability measure on $\{\pm1\}^3$
    which takes the value $t/8$ on elements of $G$ and the value $(1-t)/8$ on elements not in $G$.
    This is a continuum family of $G$-invariant probability measures.
\end{example}

\section{Median Algebras}\label{sec:medianalgebras}

This section will focus on essential results and definitions in the context of median algebras that are required for this paper.
For an extensive survey on median algebras we refer the reader to \cite{rol98, fioBoundary, bowMed}.

A  \textit{median algebra} is a set $M$ with a ternary operator $m:M^3\rightarrow M$, satisfying the following three axioms. For every $x,y,z,u,v\in M$ 
\begin{enumerate}
    \item $m(x,y,z)=m(x,z,y)=m(y,x,z)$,
    \item $m(x,x,y)=x$, and
    \item $m(m(x,y,z),u,v)=m(x,m(y,u,v),m(z,u,v))$
\end{enumerate}
A \textit{median morphism} between two median algebras $M$ and $N$ is a map $\phi:M\rightarrow N$ such that $\phi\circ m=m\circ \phi^3$. For two points $x$ and $y$ in $M$, the \textit{interval} $[x,y]$ is defined as the set of all points $z\in M $, such that $m(x,y,z)=z$. 
A subset $C\subset M$ is said to be \textit{convex}, if for every $x,y\in C$, $[x,y]\subset C$. 
The following is a median algebra version of Helly's theorem:

\begin{lemma}[Helly's Theorem, {\cite[Theorem 2.2]{rol98}}] \label{lem:helly}
 If $C_1,...,C_n$ are convex, and for every $i\neq j,\ \ C_i\cap C_j\neq \varnothing$, then also $\overset{n}{\underset{i=1}{\cap}}C_i \neq \varnothing$.
\end{lemma}

The \textit{convex hull} of a subset $S\subset M$, denoted by $\Conv(S)$, is the smallest convex set containing S.
For two subsets $A,B\subset M$, the \textit{join} of A and B, denoted by $[A,B]$, is the union of all intervals $[a,b]$, for $a\in A$ and $b\in B$. 
For two convex sets, the join gives a neat description for the convex hull of their union. Namely:

\begin{lemma}[{\cite[Proposition 2.3]{rol98}}] \label{lem:convexisjoin}
    If $C,C'\subset M$ are two convex sets, then, $\Conv(C\cup C')=[C,C']$
\end{lemma}

A \textit{half-space} $\hs\subset M$, is a convex subset, such that the set $\hs^*:=M\backslash \hs$ is also convex.
An unordered pair of two complementary half-spaces $\mathfrak{w}=\{\hs,\hs^*\}$ is called a \textit{wall}.
Given two sets $A$ and $B$, we say that a half-space $\hs$ is \textit{separating A from B}, if $A\subset\hs$ and $\hs\cap B=\varnothing$. 
We denote by $\Delta(A,B)$ the collection of all half-spaces that separates $A$ from $B$.
We say the a wall $\mathfrak{w}$ \textit{separates} $A$ and $B$ if $\mathfrak{w}\cap \Delta(A,B)$ is non-empty.
One of the main features of median algebras is that for any two convex sets $A$ and $B$, $\Delta(A,B)\neq \varnothing$. See theorem 2.8 in \cite{rol98}.

Let $C\subset M$ be a subset and $x\in M$. An element $y\in C$ is said to be the \textit{gate} of $x$ in $C$, if $y\in [x,z]$ for every $z\in C$.
Note that the gate, if exist, is unique. 
We say that $C$ is \textit{gate-convex}, if for every $x\in M$ there exists a gate $y$ in $C$.
Given a gate-convex set $C$, we denote by $\pi_C$ the map that assign for every element $x\in M$ its gate in $C$.
The map $\pi_C$ is called \textit{gate-projection}.
The following are key features of gate-convex sets:

\begin{lemma}[{\cite[Proposition~2.1]{fioBoundary}}] \label{lem:fio1}
    If a map $\phi:M\rightarrow M$ is a gate-projection, then for all $x,y,z\in M$, we have $\phi(m(x,y,z))=m(\phi(x),\phi(y),\phi(z))=m(\phi(x),\phi(y),z)$.
\end{lemma}

\begin{lemma}[{\cite[Lemma~2.2]{fioBoundary}}] \label{lem:fio2}
\begin{enumerate}
    \item If $C_1\subset M$ is convex and $C_2\subset M$ is gate-convex, the projection $\pi_{C_2}(C_1)$ is convex. If moreover, $C_1\cap C_2\neq \varnothing$, we have $\pi_{C_2}(C_1)=C_1\cap C_2$.
    \item If $C_1,C_2\subset M $ are gate-convex, the sets $\pi_{C_1}(C_2)$ and $\pi_{C_2}(C_1)$ are gate-convex with gate-projections $\pi_{C_1}\circ \pi_{C_2}$ and $\pi_{C_2}\circ \pi_{C_1}$, respectively.
    \item If $C_1, C_2\subset M$ are gate-convex and $C_1\cap C_2\neq \varnothing $, then $C_1\cap C_2 $ is gate-convex with gate-projection $\pi_{C_2}\circ \pi_{C_1}=\pi_{C_1}\circ \pi_{C_2}$. In particular, if $C_2\subset C_1$, then $\pi_{C_2}=\pi_{C_1}\circ \pi_{C_2}$.
\end{enumerate}
\end{lemma}

It follows by the axioms of the median operator that for every $x,y,z\in M$, $m(x,y,z)\in [x,y]$. 
Moreover, by lemma \ref{lem:fio1} and axiom (3), intervals are gate-convex, and gate-convex sets are convex. 
A criterion for a convex set to be gate-convex, will be given in the context of compact median algebras.

Let $\phi:M\rightarrow M $ be an automorphism of M, and $C$ a gate-convex set. Observe that for every $x\in M$, and every $y\in C$, we have
$$\phi(\pi_C(x))=\phi(m(x,y,\pi_C(x)))=m(\phi(x),\phi(y),\phi(\pi_C(x)))\in [\phi(x),\phi(y)].$$
That is, $\phi(C)$ is gate-convex, and for every $x\in M$, $\pi_{\phi(C)}(x)=\phi(\pi_C(\phi^{-1}(x))
)$.

A \textit{topological median algebra} is a median algebra M, endowed with a Hausdorff topology, with respect to which the median operator is continuous. 
It is called \textit{locally open-convex} if, in addition, each of its points has a basis of open and convex neighborhoods.
The reader should note the differences between our definition of locally open-convex to that of \textit{locally convex} in \cite{fioBoundary}. 

Let $M$ be a compact median algebra. 
For any two points $x,y\in M$, the interval $[x,y]$ is closed, as it is the continuous image of the compact set $\{x\}\times\{y\}\times M$. 
It follows by lemmata 2.6 and 2.7 in \cite{fioBoundary}, that any convex set in $M$ is gate-convex if and only if it is closed, and that gate-projections are continuous.
In particular, the convex hull of two gate-convex sets $C$ and $C'$ is compact, and therefore a gate-convex as well, as it is, by lemma \ref{lem:convexisjoin}, the continuous image of the compact set $C\times C' \times M$.

\begin{lemma}[{\cite[Lemma 2.4]{BaderTaller}}]\label{lem:gateconvexseparbypoint}
    For disjoint non-empty subsets $A,B\subset M$, if $A$ is gate-convex then there exists $a\in A$ such that $\Delta(A,B)=\Delta(a,B)$.
\end{lemma}

Note that it implies in particular that also $\Delta(B,A)=\Delta(B,a)$.

Our main object is \textit{second countable, locally open-convex, compact}, or in short \textit{sclocc}, median algebra M. 
We fix such a space for the rest of this section. 

Denote by $\Prob(M)$ the space of probability measures on $M$. 
In \cite{BaderTaller} we introduce the \textit{self-median operator} $\Phi: \Prob(M) \rightarrow \Prob(M)$, given by $\mu\mapsto m_*(\mu^3)$. 
A probability measure is said to be \textit{balanced}, if it is invariant under $\Phi$. 

Given a cube $C=\{0,1\}^k\simeq (\bbZ/2\bbZ)^k$, one can easily verify that its corresponding normalized Haar measure $\lambda$, is balanced. 
Thus, each subcube $C\subset M$, induces a balanced measure on M. 
Such a measure in $M$ is called \textit{cubical}.
We have the following characterisation of balanced measures:

\begin{theorem}[{\cite[Theorem A]{BaderTaller}}]\label{thm:balancediscubical}
    Every balanced measure on a sclocc median algebra is cubical.
\end{theorem}

A half-space $\hs$ in $M$ is said to be \textit{admissible} if $\hs$ is open and $\hs^*$ has non empty interior.
The following is an enhanced separation property: 

\begin{lemma}[{\cite[Proposition 3.6]{BaderTaller}}]\label{lem:separatewithadmissible}
    Let $C$ and $C'$ be two disjoint closed convex subsets in M. Then, there exist an admissible half-space $\hs\in \Delta(C,C')$.
\end{lemma}

\begin{corollary}\label{cor:separatingcollection}
    There exists a countable collection $\Hc$ of closed half-spaces in M, such that for every $x\neq y \in M$, there exist two disjoint, closed half-spaces $\hs\in \Delta(x,y)\cap \Hc$ and $\hs'\in \Delta(y,x)\cap\Hc$, and such that $x\in \hs^{\circ}$ and $y\in \hs'^{\circ}$. 
\end{corollary}
\begin{proof}
 As $M$ is second countable and locally open-convex, we can find a countable basis for the topology, $\mathcal{C}$, consisting of open convex sets.
Thus, every two different points can be separated by two elements of $\mathcal{C}$ with disjoint closures. 

Let $U$ and $V$ be two elements of $\mathcal{C}$ with disjoint closures.
By lemma \ref{lem:separatewithadmissible} we can choose a closed half-space $\hs_{(U,V)}\in \Delta (\overline{U},\overline{V})$. 
Using lemma \ref{lem:separatewithadmissible} again, we find closed half-space $\hs_{(V,U)}\in \Delta (\overline{V},\hs_{(U,V)})$.

Denote by $\Hc$ the countable collection of all half-spaces of the form $\hs_{(U,V)}$ and $\hs_{(V,U)}$.
The reader can now verify that $\Hc$ possesses the required properties.   
\end{proof}

Given a collection of half-spaces $\Hc$, and a non-empty subset $A$ of $M$, we denote by $\Hc(A)$ the subcollection of $\Hc$ consisting of all the half-spaces $\hs\in \Hc$ that cut $A$ into two, i.e. such that both $\hs\cap A$ and $\hs^*\cap A$ are non empty.

Fix some collection of half spaces $\Hc$, and its corresponding collection of walls $\Wc_{\Hc}$. That is,
$$\mathcal{W}_{\Hc}:=\big\{\{\hs,\hs^*\}\ | \ \hs\in \Hc\big\}.$$

We have a natural surjection $\Hc\rightarrow \Wc_{\Hc}$, given by $\hs\mapsto \{\hs,\hs^*\}$. 
Fix a section $\sigma:\Wc_{\Hc}\rightarrow \Hc $. 
Given a set $A$ we denote by $\chi_A$ its characteristic function.
Consider the following map
$$\iota_{\Wc_{\Hc}}:M\rightarrow 2^{\Wc_{\Hc}}, \ x\mapsto\big(\chi_{\sigma(\mathbf{w})}(x)\big)_{\mathbf{w}\in \Wc_{\Hc}}.$$

We say that $\Wc_{\Hc}$ is \textit{separating} if $\iota_{\Wc_{\Hc}}$ is injective, that is, if the walls in $\Wc_{\Hc}$ are separate points in M.
It is called \textit{transverse} if $\iota_{\Wc_{\Hc}}$ is a surjection. 
The collection $\Hc$ is said to have one of the above properties if $\Wc_{\Hc}$ possess it.
We note that these properties are independent of the choice of the section $\sigma$.

We say that two collections of walls $\Wc_1$ and $\Wc_2$ are transverse, if for every $\mathfrak{w}_1\in\Wc_1$ and $\mathfrak{w}_2\in\Wc_2$, the set of walls $\{\mathfrak{w}_1,\mathfrak{w}_2\}$ is transverse.
Two collections of half-spaces $\Hc_1$ and $\Hc_2$ are said to be transverse, if their collections of walls, $\Wc_{\Hc_1}$ and $\Wc_{\Hc_2}$, are transverse.

We finish with a handy observation:
\begin{lemma}\label{lem:bijectioniscontinuous}
    Let $\phi:M\rightarrow N$ be a continuous median bijection between two sclocc median algebras. Then $\phi$ is an isomorphism.
\end{lemma}
\begin{proof}
    Any continuous bijection between compact Hausdorff spaces is a homeomorphism. So it is left to show that the inverse $\phi^{-1}$ is a median map.

Indeed, let $x,y,z\in M$. Then,

\begin{align*}
    \phi^{-1}(m(x,y,z))&=\phi^{-1}(m(\phi(\phi^{-1}(x)),\phi(\phi^{-1}(y)),\phi(\phi^{-1}(z))))\\ &=\phi^{-1}\circ \phi \big(m(\phi^{-1}(x),\phi^{-1}(y),\phi^{-1}(z))\big)=m(\phi^{-1}(x),\phi^{-1}(y),\phi^{-1}(z)).
\end{align*}

\end{proof}

\section{Cubing Decomposition And Proof Of Proposition \ref{prop:structure}} \label{sec:cubing}
Given a median algebra M, denote by $\Hc_M$ and $\Wc_M$ the collection of all half-spaces and walls of M, respectively.
We fix a sclocc median algebra M and a group $G$ as in Proposition \ref{prop:structure}, and a section $\sigma_M:\Wc_M\rightarrow \Hc_M$.

\begin{lemma}\label{lem:productthm}
    Let $\Wc\subset \Wc_M$ be a $G$-invariant subcollection of walls that separates points in M.
    Then, there exist two sclocc median algebras $M'$ and $C$, where C is a finite cube and an isomorphism $\phi:M\rightarrow M'\times C$, if and only if there exists a partition $\Wc=\Wc_1\amalg \Wc_2$ such that the following conditions hold:
    \begin{enumerate}
        \item The set $\Wc_1$ is transverse, in which all the walls consists of clopen half-spaces,
        \item The collections $\Wc_1$ and $\Wc_2$ are transverse.
    \end{enumerate}

    Furthermore, we can choose $M'$ and $C$ to be $G$-spaces, and $\phi$ to be $G$-equivariant, if and only if there exists a partition as above, with $\Wc_1$ and $\Wc_2$ $G$-invariant.
\end{lemma}
\begin{proof}
    Suppose first that we have a partition as above. 
    Then, item (1) implies that the map $\iota:=\iota_{\Wc_1}$ is in fact a continuous projection onto the cube $C:=2^{\Wc_1}$.
    Fix $x\in C$, and denote by $M_x$ the fiber $\iota^{-1}(x)$.
    As $\iota$ is a continuous median map, $M_x$ is closed and convex, thus, a gate-convex set.
    We claim that for every $x,y\in C$, $\pi_{M_x}|_{M_y}$ is an isomorphism.
    
    Let $a,b\in M_y$. 
    By lemma \ref{cor:separatingcollection}, we can find two disjoint closed half-spaces $a\in \hs_a^0\subset\hs_a\in\Delta(a,b)$ and $b\in\hs_b^0\subset \hs_b\in\Delta(b,a)$.
    For a set $A\subset M$, denote by $\overline{A}$ the topological closure of $A$.
    Applying lemma \ref{lem:gateconvexseparbypoint} to the pairs $\{a,\overline{\hs_a^*}\cap M_y\}$ and $\{b,\overline{\hs_b^*}\cap M_y\}$, implies the existence of two points $a'\in \overline{\hs_a^*}\cap M_y$ and $b'\in\overline{\hs_b^*}\cap M_y$, such that $\Delta(a,a')=\Delta(a,\overline{\hs_a^*})$ and $\Delta(b,b')=\Delta(b,\overline{\hs_b^*})$.
    These equalities are true in M, although we applied the lemma in $M_y$, thanks to lemma 2.3 in \cite{fioBoundary}.

    Denote by $\Hc$ the union $\sigma_M(\Wc)\cup \sigma_M(\Wc)^*$.
    As $\Wc$ is separating, we can find two half-spaces $\fs_a\in \Delta(a,a')\cap \Hc$ and $\fs_b\in\Delta(b,b')\cap \Hc$.
    Now, it follows by Helly's theorem, compactness, and the assumption that $\Wc_1$ and $\Wc_2$ are transverse, that $\overline{\fs_a}\cap M_x\neq\varnothing$ and that $\overline{\fs_b}\cap M_x\neq\varnothing$.
    Note that $\overline{\fs_a}$ and $\overline{\fs_b}$ are disjoint, and therefore, it follows by lemma \ref{lem:fio2}(1) that $\pi_{M_x}(a)\neq \pi_{M_x}(b)$. 
    Thus, $\pi_{M_x}|_{M_y}$ is injective.
    By lemma \ref{lem:fio2}(2), $\pi_{M_x}\circ\pi_{M_y}|_{\pi_{M_x}(M_y)}=id_{\pi_{M_x}(M_y)}$. 
    In particular, $\pi_{M_y}(\pi_{M_x}(M_y))=\pi_{M_x}(M_y)$, and thus, by injectivity, $\pi_{M_y}(\pi_{M_x}(M_y))=M_y$, that is, $\pi_{M_y}|_{M_x}$ is surjective.
    The claim now follows by the symmetry of the arguments.

    Fix $x_0\in C$, denote $M_0:=M_{x_0}$ and by $\pi_0$ its gate-projection.
    As $M$ is the disjoint union of the $M_x$, it is clear that the continuous map $\iota\times \pi_0:M\rightarrow C\times M_0$ is a bijection.
    It follows by lemma \ref{lem:bijectioniscontinuous} that this is an isomorphism, as needed.
    Note that $C$ must be finite, as the intervals in M are countable.

    Suppose now that $M=M'\times C$, for $M'$ and $C$ as in the lemma.  
    Let $p_{M'}$ and $p_C$ be the natural projection to each factor.
    We have the following two subcollections of $\Wc_M$:
    \begin{align*}
        & \Wc_M^1:=\big\{\{p_C^{-1}(\ws)\ | \ \ws\in\Wc_C\big\} \\
        & \Wc_M^2:=\big\{\{p_{M'}^{-1}(\ws)\ | \ \ws\in\Wc_{M'}\big\} 
    \end{align*}
    Clearly, these collections are transverse, and since $\Wc_C$ is transverse and its walls are consists of clopen half-spaces, so is $\Wc_M^1$ and so are its walls. 
    Take $\hs\in\Hc_M$. 
    Observe that if $(x_1,y_1),(x_2,y_2)\in \hs$, then, $[x_1,x_2]\times[y_1,y_2]\subset \hs$. 
    Hence, $\hs$ is of the form $C_1\times C_2$ for some two convex sets $C_1\subset M'$ and $C_2\subset C$.
    Since this is true also for $\hs^*$, we must have $\mathfrak{h}\in \Wc_M^1$ or $\mathfrak{h}\in \Wc_M^2$. That is, $\Wc_M^1$ and $\Wc_M^2$ constitute a partition of $\Wc_M$.
    Define $\Wc_1:=\Wc\cap\Wc^1_M$ and $\Wc_2:=\Wc\cap\Wc^2_M$.
    It is clear that this is the required partition.
    We note that it is possible that $\Wc_1$ is in fact empty, which just implies that $C$ is a point.

    The last assertion in the lemma is immediate.
\end{proof}

\begin{proof}[Proof of proposition \ref{prop:structure}]
    If M is already a cube, there is nothing to prove.
    Otherwise, define 
    \begin{align*}
        & \Wc_1:=\big\{\ws\ | \ \ws\in\Wc_M \text{ consist of clopen half-spaces and }\Wc_M\backslash \{\ws\} \ and \ \{\ws\} \text{ are transverse }\big\} \\
        & \Wc_2:=\Wc_M\backslash \Wc_1
    \end{align*}

    Any isomorphism of $M$ must preserve $\Wc_1$, and therefore, is also preserving $\Wc_2$.
    By construction, these sets meets the conditions in lemma \ref{lem:productthm}, and thus, correspond to a decomposition $M=M'\times C$, where $M'$ is a sclocc median algebra, $C$ is a finite cube, and on both there exist $G$-actions, such that the action on $M$ is the diagonal.

    Let $\varphi :M\rightarrow C'$ be a surjective morphism, with $C'$ a cube. 
    Denote by $\Wc$ the collection of walls of the form $\{\varphi^{-1}(\fs),\varphi^{-1}(\fs^*)\}$ for $\fs\in\Hc_{C'}$.
    It is clear that $\Wc\subset \Wc_1$, and that $\varphi$ is the composition of $p_C=\iota_{\Wc_1}$ with the projection from $C$ to the factor generated by $\Wc$.
\end{proof}

\section{Special Constructions And Technical Lemmata} \label{sec:constructions}
The purpose of this section is to establish all the constructions and technical results that are necessary for the proof of the main theorems, e.g. checking the measurability of some sets and maps between them, etc.
The trusting reader can skip this section for the benefit of the reading flow.

\subsection{Subsets of $\mathcal{C}$M}
Let M be a \textit{sclocc} median algebra. 
As it is a metrizable space, we may consider its space $\Cc M$ of closed subsets, endowed with the Hausdorff distance that corresponds to a metric on M. 
We fix such a metric and recall that $\Cc M$ is a compact metric space (see for example lemma 5.31 in \cite{Bridson1999}).

Let $F\subset M$ be a finite subset.
It follows by lemma \ref{lem:convexisjoin}, and by an inductive argument, that the convex hull of F is the join of two compact convex sets, and therefore, it is itself a compact and convex set.
In particular, as discussed in the previous section, $conv(F)$ is a gate-convex set, and is an element of $\mathcal{C}M$.

\begin{definition}\label{def:joinspace}
Denote by $\Jc_n$ the subspace of $\mathcal{C}M$ consisting of all the convex sets that generated by n elements. In addition, denote by $\Jc_n^{\subseteq}$ the subspace of $\Jc_n\times\Jc_n$ consisting of pairs (C,J) such that $J\subset C$.
\end{definition}

\begin{lemma}\label{lem:joinclosed}
$\Jc_n$ and $\Jc_n^{\subseteq}$ are closed subspaces of $\mathcal{C}M$ and $\mathcal{C}M\times \mathcal{C}M$, respectively.
\end{lemma}

\begin{proof}
Notice that it is enough to show that $\Jc_n$ is a closed subspace. We argue by induction on n. The case $n=1$ is clear, as $\Jc_1$ is none other than M itself. 
Assume the claim is correct for n, and let $C_i\in\Jc_{n+1}$ be a sequence that converges to some element $C\in \mathcal{C}M$. 
For every i, denote by $x_j^i\in M$, j=1,..,n+1, the collection of generators of $C_i$, that is, $C_i=Conv(\{x_j^i\}_{j=1}^{n+1})$. 
For every i, denote by $C_i'$ the set $Conv(\{x_j^i\}_{j=1}^n)$, and notice that by the definition of the convex hall,  $C_i=Conv(C_i'\cup \{x_{n+1}^i\})$. 
By the induction hypothesis, after taking a subsequence, we may assume that there exist $x_{n+1}\in M$ and $C'\in\Jc_{n}$ such that the sequences $\{x_{n+1}^i\}_i$ and $\{C_i'\}_i$ converge to $x_{n+1}$ and $C'$, respectively. 
We claim that $C=Conv(C'\cup \{x_{n+1}\})\in \Jc_{n+1}$.

Lemma \ref{lem:convexisjoin} implies that $C_i=[C_i',x_{n+1}^i]$ and $Conv(C'\cup \{x_{n+1}\})=[C',x_{n+1}]$. 
Notice that, as for every i, $C_i'\subset C_i$, then $C'\subset C$. 
Therefore, we get one inclusion for free: $[C',x_{n+1}]\subset C$. 
For the other inclusion, fix an element $y\in C$. Let $y_i\in C_i$ be a sequence that converges to $y$. 
By the observation above, for every i we can find $z_i\in C_i'$ such that $y_i\in [z_i,x_{n+1}^i]$. After passing to a subsequence, if needed, we may assume that $z_i$ converges to an element $z\in C'$. 
But now, by continuity of the median
$$y=\underset{n}{\lim}\ y_i=\underset{n}{\lim}\ m(z_i,y_i,x_{n+1}^i)=m(z,y,x_{n+1})\in [C',x_{n+1}]$$ 
\end{proof}

Denote by $\Cub(M)$ the collection of all closed subcubes in M. 

\begin{lemma}\label{lem:cubclosed}
    The space $\Cub(M)$ is a closed subspace of $\mathcal{C}M$.
\end{lemma}

Before proving this lemma, let us introduce a new notion, and a related property. Let $A\subset M$. We denote by $\Ends(A)$ the subset of all elements $x\in A$, for which there is an \textit{antipodal element in A}. 
That is, an element $x^*\in A$, such that $m(x,x^*,z)=z$, for all $z\in A$. If $A\in \Jc_2$, this amounts to saying that $[x,x^*]=A$. 
We claim the following:

\begin{lemma}\label{lem:suparecubes}
    Let $A\in \mathcal{C}M$ be a median subalgebra. If $\Ends(A)\neq \varnothing$, then $\Ends(A)\in \Cub(M)$
\end{lemma}
\begin{proof}
    First, let us prove that $\Ends(A)$ is a closed median subalgebra. 
    Let $x_1,x_2,x_3\in \Ends(A)$, and let $x_1^*,x_2^*,x_3^*\in \Ends(A)$ be their antipodal elements in $A$. 
    We claim that $m(x_1^*,x_2^*,x_3^*)$ is the antipodal in $A$ of $m(x_1,x_2,x_3)$. 
    Namely, that for every $z\in A$ $$(*) \ m\big(z,m(x_1,x_2,x_3),m(x_1^*,x_2^*,x_3^*)\big)= z$$
    
    Fix $z\in A$. 
    For the sake of contradiction, suppose that we can find $$\hs\in \Delta\big(z,m\big(z,m(x_1,x_2,x_3),m(x_1^*,x_2^*,x_3^*)\big)\big)$$

    Since $z\in \hs$, it follows by convexity of $\hs$ and $\hs^*$, that both $m(x_1,x_2,x_3)$ and $m(x_1^*,x_2^*,x_3^*)$ are in $\hs^*$. 
    By convexity, again, there must be $3\geq i\geq 1$ such that both $x_i$ and $x_i^*$ are in $\hs^*$. 
    But now, $$z\in A=[x_i,x_i^*]\cap A\subset \hs^*$$
    which is a contradiction. 
    Hence, $(*)$ holds. 
    
    Suppose now that there is a sequence $x_n\in \Ends(A)$ that converges to some element $x\in A$. 
    Let $x^*_n\in \Ends(A)$ be the corresponding sequence of antipodal elements. 
    After replacing $x_n$ with some sub-sequence, if necessary, we may assume that $x_n^*$ converges to an element $x^*\in A$. 
    Fix $z\in A$. 
    By continuity of the median operator,
    \begin{align*}
        z & = \underset{n\rightarrow \infty}{\lim} z = \underset{n\rightarrow \infty}{\lim} m(x_n,x_n^*,z) \\
        & = m(x,x^*,z)
    \end{align*}
    Thus, $A=[x,x^*]\cap A$, which implies that $x\in \Ends(A)$, as needed.

    We now show that $\Ends(A)$ isomorphic to a cube. 
    By the above, we may assume without loss of generality that $\Ends(A)=A=M$.
    Denote by $\mathcal{W}^{\circ}$ the collection of all clopen walls of $M$.
    That is, the collection of all pairs $\{\hs,\hs^*\}$ for a clopen half-space $\hs\subset M$.

    It follows by corollary 3.4 in \cite{BaderTaller}, that it is sufficient (and necessary) to show that $\mathcal{W}^{\circ}$ is separating points in $M$ and is transverse.

    For starter, let us show that for any closed half-space $\hs$, $\{\hs,\hs^*\}\in \mathcal{W}^{\circ}$. Notice that for every $x\in M$, the antipodal element $x^*$ is unique. 
    Indeed, if $x^*$ and $x'^* $ are antipodal elements of $x$, then $$x^*=m(x,x^*,x'^*)=x'^*$$
    Moreover, the above discussion implies that if $x_n$ converges to $x$, then $x_n^*$ converges to $x^*$. 
    Thus, the antipodal map, $*:x\mapsto x^*$, is a continuous median automorphism of $M$.
    If $\hs$ is a non-empty half-space, such that also $\hs^*$ is non-empty, then $\hs=*^{-1}(\hs^*)$ (otherwise it would imply that there exists $x\in M$, such that $M=[x,x^*]\subset \mathfrak{h}$ or $M=[x,x^*]\subset \mathfrak{h}^*$). 
    In particular, if $\hs$ is closed, then it is also open, as needed.

    Let $\{\hs,\hs^*\}$ and $\{\fs,\fs^*\}$ be two walls in $\mathcal{W}^{\circ}$.
    Suppose that they are not transverse.
    Without loss of generality we may assume that $\hs\subset \fs$. 
    But then, $\hs^*= *(\hs)\subset *(\fs)=\fs^*\subset \hs^*$ which implies that $\{\hs,\hs^*\}=\{\fs,\fs^*\}$.
    That is, $\mathcal{W}^{\circ}$ is transverse, as required.
\end{proof}

\begin{corollary}\label{cor:suppintiscub}
    For $I\in\Jc_2$, $\Ends(I)\in \Cub(M)$.
\end{corollary}

In the course of proving lemma \ref{lem:suparecubes}, we also proved the following lemma:
\begin{lemma}\label{lem:clopenhalfspace}
    Suppose that $M$ is a \sclocc median algebra which is a cube. Then, every closed half-space is also open. 
\end{lemma}

\begin{proof}[Proof of lemma \ref{lem:cubclosed}]
    Let $C_n\in \Cub(M)$ be a sequence that converges to some closed set $C\in \mathcal{C}M$. 
    It follows by lemma \ref{lem:suparecubes} that we only need to show that $C$ is a median subalgebra and that $\Ends(C)=C$.

    If $x_n,y_n$ and $z_n\in C_n$ converges to $x,y$ and $z\in C$, respectively, then clearly, $m(x_n,y_n,z_n)\in C_n$ converges to $m(x,y,z)\in C$.
    
    Pick $x\in C$, and $x_n\in C_n$ such that $x_n$ converges to $x$. 
    Let $x_n^*\in C_n$ be the corresponding sequence of antipodal elements, and suppose it converges to $x^*\in C$ (restrict to a sub-sequence, if necessary).
    Fix $z\in C$, and some (sub-)sequence $z_n\in C_n$ that converges to $z$. 
    Then, 
    
    \begin{align*}
        z & = \underset{n\rightarrow \infty}{\lim} z = \underset{n\rightarrow \infty}{\lim} m(x_n,x_n^*,z) \\
        & = m(x,x^*,z)
    \end{align*}
\end{proof}

\subsection{Mapping Properties}
For the rest of this section, fix a \textit{sclocc} median algebra M, and a discrete group $G$ that acts on M continuously by automorphisms.
Note that in particular, we are not assuming that intervals in M are countable, nor that subcubes of M are of finite dimension (unless stated otherwise).

\begin{lemma}\label{lem:continousmapping}
    The map $\Jc_n^2\rightarrow \Jc_n^{\subseteq}$, defined by $(I,J)\mapsto (I,\pi_I(J))$, is continuous. 
\end{lemma}
\begin{proof}
    This is simply because the map $\pi_I$ is a continuous function, that depends continuously on the generators of I.  
\end{proof}




\begin{lemma}\label{lem:InducedAction}
    Let $C\subset M$ be a gate convex set. Suppose that for every $\g\in G$, $\pi_C$ is a surjection from $\g C$ to C. Then the assignment $\g.x:= \pi_C(\g x)$, for $\g\in G$ and $x\in C$, defines a continuous action of $G$ on C by median automorphisms. Moreover, $\pi_C$ is a median $G$-map between $M$ to $C$.
\end{lemma}
\begin{proof}
    The continuity part is clear, as the new action is defined by the composition of continuous functions. 
    We now show that this is an action. Let $x\in C$ and $\g',\g\in G$. By lemma \ref{lem:fio2}.(2), as $\pi_C|_{\g C}$ is surjective, $\pi_C\circ \pi_{\g C}=\pi_C$. So we have: 
$$\g.(\g'.x)=\pi_C(\g\pi_C(\g'x))=\pi_C(\pi_{\g C}(\g\g'x))=\pi_C(\g\g'x)=(\g\g').x$$ 
As required.

If $y,z\in C$ are another two elements, then by lemma \ref{lem:fio1} we have
\begin{align*}
  \g.m(x,y,z)&=\pi_C(\g m(x,y,z))=\pi_C(m(\g x, \g y, \g z))\\
    &=m(\pi_C(\g x),\pi_C(\g y), \pi_C(\g z)) = m(\g .x,\g .y,\g .z)  
\end{align*}
Finely, note that by the definition of the induced action, $\pi_C$ is a $G$-map, and it is a median map by lemma \ref{lem:fio1}.
\end{proof}

\begin{lemma}\label{lem:cubetoint}
    The map $C\mapsto I_C:=\Conv(C)$, is a continuous $G$-map from $\Cub(M)$ to $\Jc_2$.
\end{lemma}
\begin{proof}
    The fact that this map is a $G$-map, is clear. We left to prove that it is continuous.
    
    Let $C_n$ be a sequence of subcubes that converges to $C\in \Cub(M)$. 
    Fix some $x\in C$, its antipodal $x^*\in C$, some sequence $x_n\in C_n$ that converges to $x$, and the corresponding sequence of antipodal points $x_n^*\in C_n$.
    We recall that as in the proof of lemma \ref{lem:suparecubes}, it follows that $x_n^*$ converges to $x^*$.

    Fix $z\in I_C$, and denote $z_n:=m(x_n,x_n^*,z)\in I_{C_n}$.
    By the continuity of the median, we have
    $$\underset{n}{\lim} \ z_n=\underset{n}{\lim} \ m(x_n,x_n^*,z)=m(x,x^*,z)=z$$

    On the other hand, let $z_n\in I_{C_n}$ be a (sub-)sequence that converges to some $z\in M$.
    Then, again, by contiuity of the median:
    \begin{align*}
        z & = \underset{n}{\lim} \ z_n = \underset{n}{\lim} \ m(x_n,x_n^*,z_n) \\
        & = m(x,x^*,z)\in I_C
    \end{align*}
    Thus, indeed, $I_{C_n}$ converges to $I_C$.
\end{proof}

It is natural to wonder whether the map in the other direction, namely, the map that assigns to every interval its collection of ends, is continuous as well. 
Apparently, the answer is no.
This is reflected by the fact that the image of the $\Ends$ map is not always a closed subsets of $\Cub(M)$. 

This image, on the other hand, can be identified as the collection of maximal subcubes.
Here, we mean that $C$ is maximal, in the the sense that if $C'$ is another subcube containing $C$, than $C$ is a proper face of $C'$.
Denote this collection by $\Cub_m(M)$.

\begin{lemma}\label{lem:maximalcubesareGdelta}
    The set $\Cub_{m}(M)$ is a $G_{\delta}$ subset of $\Cub(M)$.
\end{lemma}
\begin{proof}
    Define $F_{\epsilon}:=\{ C\in \Cub(M) \ | \ d_H(C,\Ends(I_C))\geq \epsilon\}$, where, as always, $I_C$ stands for the convex hull of C. 
    When $\epsilon>0$, then, $F_{\epsilon}\subset \Cub(M)\backslash \Cub_m(M)$.
    On the other hand, if $C\notin \Cub_m(M)$, then $\epsilon':=d_H(C,\Ends(I_C))>0$, and in particular, $C\in F_{\epsilon'}$. 
    Thus, we have 
    $$\Cub(M)\backslash \Cub_m(M) = \underset{n\in\bbN}{\bigcup}F_{1/n}$$
    where the right hand side is an increasing union of subsets. 
    Therefore, it is enough to show that $F_{\epsilon}$ is a closed set.

    Fix some sequence $C_n\in F_{\epsilon}$, that converges to some $C\in \Cub(M)$. 
    After restricting to a subsequence, if necessary, we may assume that $C_n':=\Ends(I_{C_n})$ converges to some subcube $C'$. We have the following short sequence of inclusions:
    $$C\subset C'\subset \Ends(I_C)=\Ends(I_{C'})$$
    We note that the later equality follows from the fact that $I_{C_n}$ converges to $I_C$. See lemma \ref{lem:cubetoint}. 
    Therefore, for any limit $x\in I_C$ of some $x_n\in \Ends(I_{C_n})$, there exists an antipodal point $x^*$, which is the limit of $x_n^*$.
    Thus, $x\in \Ends(I_C)$.
    
    For every n, the fact that $d_H(C_n,\Ends(I_{C_n}))\geq \epsilon$, implies that there exists $x_n\in \Ends(I_{C_n})$ with $d(\{x_n\},C_n)\geq \epsilon$, that is, for every $y_n\in C_n$, $d(x_n,y_n)\geq \epsilon$. 
    After further restricting to another sub-sequences, if necessary, we may assume that $x_n$ converges to some $x\in C'$.
    Take $y\in C$, and $y_n\in C_n$ that converges to $y$. Then by continuity of the metric, we have $$\epsilon \leq \underset{n}{\lim} \ d_H(\{x_n\},\{y_n\}) = d_H(\{x\},\{y\})=d(x,y)$$
    In particular,
    $$\epsilon \leq d_H(C,C')\leq d_H(C,\Ends(I_C))$$
    Hence, $C\in F_{\epsilon}$, as required.
    
\end{proof}

\begin{corollary}\label{cor:inttocubeismeasurable}
    The assignment  $I\mapsto \Ends(I)$, is a Borel $G$-map isomorphism from $\Jc_2$ to $\Cub_m(M)$.
\end{corollary}
\begin{proof}
    Denote the above assignment by $\Ends$. 
    This is a bijection onto $\Cub_m(M)$, with an inverse denoted by $\Conv$.

    We note that a $G_{\delta}$ subset of a Polish space, is Polish.
    See \cite[Theorem 3.11]{Kechris1995}.
    Therefore, by lemmata \ref{lem:cubetoint} and \ref{lem:maximalcubesareGdelta}, $\Conv$ is a continuous bijection between two Polish spaces.
    
    The corollary now follows by the fact that such a map is a Borel isomorphism. See \cite[Theorem 15.1]{Kechris1995}.
\end{proof}

\begin{lemma}\label{lem:cubetobalanced}
    The map from $\Cub(M)$ to $\Prob(M)^{\Phi}$, defined by the rule $C\mapsto \lambda_C$, where $\lambda_C$ is the balanced measure on $C$, is a homehomorphism (with respect to the $weak^*$ topology on $\Prob(M)^{\Phi}$).
\end{lemma}

Before proving this Lemma we need the following:

\begin{lemma}\label{lem:balancemeasurefiniteintersection}
    Let $\eta\in\Prob(M)^{\Phi}$, and let $\hs_1,...,\hs_m$ be a collection of closed half-spaces. Then, $\eta(\cap\hs_i)\in \{0\}\cup\{2^{-s}\ | \ m\geq s\geq 0\}$
\end{lemma}
\begin{proof}
    Without loss of generality, suppose that $M=\supp(\eta)$.
    Denote by $\chi_{\hs_i}$ the indicator function of $\hs_i$.
    By lemma \ref{lem:clopenhalfspace}, each $h_i$ is clopen.
    Therefore, the map $\varphi:M\rightarrow \{0,1\}^m$, mapping each x to $\{\chi_{\hs_i}(x)\}_i$ is a continuous median map.
    It follows that $\varphi_*(\eta)$ is a balanced measure, and by theorem \ref{thm:balancediscubical},  it is the uniform measure on a subcube of $\{0,1\}^m$.
    Therefore, on each point in its support, its value is $2^{-s}$, for some natural number $m\geq s\geq 0$.
    The statement follows.
\end{proof}

\begin{proof}[Proof of Lemma \ref{lem:cubetobalanced}]
    We first claim that the assignment $\eta\mapsto \supp(\eta)$ is continuous from $\Prob(M)^{\Phi}$ to $\Cub(M)$. Indeed, suppose that $\eta_n\xrightarrow{w^*} \eta$. 

    Fix $x\in \supp(\eta)$. 
    Since every open neighborhood of $x$ is of positive $\eta$-measure, for $n>>0$, it is also of positive $\eta_n$-measure, so we can find in it some $x_n\in \supp(\eta_n)$. 
    Therefore, $\supp(\eta)$ is contained in any limit of $\supp(\eta_n)$.


     On the other hand, let $x\notin \supp(\eta)$. 
     For every $y\in \supp(\eta)$, we can find a closed half-space $\hs_y\in\Delta(x,y)$, and such that $x\in \hs_y^{\circ}$. 
     See corollary \ref{cor:separatingcollection}.
     By compactness of $\supp(\eta)$, there are $y_1,..,y_m\in\supp(\eta)$ such that $\supp(\eta)\subset \bigcup \hs^*_{y_i}$.
     Thus, the set $\bigcap \hs_{y_i}$ is an $\eta$-null compact set, so as $\eta$ is a Radon measure, we can find a continuous function $f$ with $\bigcap \hs_{y_i}\subset \supp(f)$, and $\eta(f)<2^{-m-1}$.
     It follows by lemma \ref{lem:balancemeasurefiniteintersection} that every balanced measure $\xi$ with $\xi(\bigcap \hs_{y_i})>0$, then, $\xi(\bigcap \hs_{y_i})\geq 2^{-m}$.
     For $n>>0$, $\eta_n(\bigcap \hs_{y_i})\leq\eta_n(f)<2^{-m}$, thus, $\eta_n(\bigcap \hs_{y_i})=0$.
     Since $x\in (\bigcap \hs_{y_i})^{\circ}$, $\lim \inf d(x,\supp(\eta_n))>0$.
     In particular, any limit of $\supp(\eta_n)$ is contained in $\supp(\eta)$.

     By lemma \ref{lem:cubclosed}, $\Cub(M)$ is a compact Hausdorff space. 
     Since $\Prob(M)^{\Phi}$ is a $weak^*$-closed subspace of $\Prob(M)$, it is also a compact Hausdorff space. 
     Clearly, the assignment $\eta\mapsto \supp(\eta)$ is a bijection between the above two sets (this is due to the uniqueness of the fully supported balanced measure on any cube).
     Thus, this assignment is also a homeomorphism, as required.
     
\end{proof}

\begin{lemma}[{\cite[Lemma A.1]{cfi16}}]\label{lem:probtosubhalfspace}
    Let $\Hc$ be a countable collection of closed half-spaces. Fix $i\in\{0,1/2,1\}$, and let $\Hc^i_{\eta}:=\{ \hs\in \Hc\ | \ \eta(h)=i\}$. Then, the map $\Prob(M)\rightarrow 2^{\Hc}$, defined by $\eta\mapsto \Hc^i_{\eta}$, is a measurable map with respect to the weak-* topology on $\Prob(M)^{\Phi}$.

    Therefore, also the map $N:\Prob(M)\rightarrow \bbN\cup \{\infty\}$, defined by $N(\eta)=|\Hc^i_{\eta}|$, is measurable 
\end{lemma}
\begin{proof}
    The proof of lemma A.1 in \cite{cfi16} works almost perfectly in our situation, except the fact that here, for $\hs\in\Hc$, the assignment $\eta\mapsto \eta(h)$ is measurable, and not necessarily continuous, with respect to the weak-* topology.

    Indeed, there is a sequence $1\geq f_n\in C(M)$, that converges point-wise to $\chi_{\hs}$.
    Therefore, by dominated convergence theorem, for every $\eta\in \Prob(M)$,
    $\underset{n}{\lim} \ \eta(f_n) =\eta(\hs)$.
    That is, the assignment $\eta \mapsto \eta(\hs)$ is the point-wise limit of the continuous maps $\eta\mapsto \eta(f_n)$.
    In particular, it is measurable.
\end{proof}

\begin{corollary}[{\cite[Corollary A.2]{cfi16}}]\label{cor:measurabilitysubshalfspa}
     Let $\Hc$ be a countable collection of closed half-spaces.
    The following maps are measurable:
    \begin{enumerate}
        \item $C_1: \Prob(M)\rightarrow \bbN\cup \{\infty\}$, defined by $C_1(\eta):=|\Hc_{\eta}|$
        \item $C_3: \Prob(M)\times \Prob(M)\rightarrow \bbN\cup \{\infty\}$ defined by $C_3(\eta,\rho):=|\Hc_{\eta}\Delta \Hc_{\rho}|$
    \end{enumerate}
\end{corollary}
\begin{proof}
    The measurability of all the above maps follows by lemma \ref{lem:probtosubhalfspace}, the fact that composition and product of measurable maps are measurable, and the fact that the operation $\Delta$ on subsets of $\Hc$, can be described in terms of continuous ring operation on $2^{\Hc}$. 
\end{proof}

\begin{lemma}\label{lem:convgateisomorphism}
    Let $C$ and $C'$ be two gate-convex subsets of $M$. Then, the gate projections $\pi_C|_{C'}$ and $\pi_{C'}|_C$ are isomorphism if and only if $\Hc(C)=\Hc(C')$.
\end{lemma}
\begin{proof}
    Suppose that $\pi_C|_{C'}$ and $\pi_{C'}|_C$ are isomorphism.
    Let $\hs\in \Hc(C)$. 
    It follows by lemma \ref{lem:fio2}(1), and the fact that both $\pi_C(C')\cap \hs$ and $\pi_C(C')\cap \hs^*$ are non-empty, that $\hs\in\Hc(C')$. 
    The claim now follows by symmetry of the argument.

    Suppose now that $\Hc(C)=\Hc(C')$.
    Assume that $x\in C \backslash\pi_C(C')$.
    Then, $\varnothing\neq \Delta(x,\pi_C(C'))\subset\Hc(C)\backslash \Hc(C')$, which is a contradiction.
    Thus, $\pi_C|_{C'}$ is surjective.
    By lemma \ref{lem:fio2}(2), $\pi_C=\pi_C\circ \pi_{C'}$.
    That is, $ \pi_C|_{C'}\circ \pi_{C'}|_C= \pi_C|_C=id_C$.
    By symmetry, the statement follows.
    
\end{proof}

For a tuple $(x,y,\eta)\in M^2\times \Prob(M)$, we denote by $\mu_{(x,y,\eta)}\in\Prob(M^3)$ the probability measure associated to the following functional:
$$\forall s\in C(M^3), \ s\mapsto\int_{M}s(x,m(x,y,c),c)d\eta(c)$$

\begin{lemma}\label{lem:boundarypairintoxmeasurable}
    The map $\mu_{(.)}:M^2\times \Prob(M)\rightarrow \Prob(M^3)$ is continuous.
\end{lemma}
\begin{proof}
    Suppose $(x_n,y_n,\eta_n),(x,y,\eta)\in M^2\times \Prob(M)$ are such that $(x_n,y_n,\eta_n)\xrightarrow{n}(x,y,\eta)$.
    Fix $s\in C(M^3)$ and $\epsilon>0$.
    By continuity of $m$ and $s$ and the compactness of $M^3$, for $n\gg 0$, $|s(x_n,m(x_n,y_n,c),c)-s(x,m(x,y,c),c)|<\epsilon/2$ for every $c\in M$.
    Also, for $n\gg 0$, $$\big|\int_{M}s(x,m(x,y,c),c)d\eta(c)-\int_{M}s(x,m(x,y,c),c)d\eta_n(c)\big|<\epsilon/2$$
    Therefore, for $n\gg 0$,
    \begin{align*}
        |\mu_{(x,y,\eta)}(s)-\mu_{(x_n,y_n,\eta_n)}(s)|& \leq \big|\int_{M}s(x,m(x,y,c),c)d\eta(c)-\int_{M}s(x,m(x,y,c),c)d\eta_n(c)\big|  \\
       &  + \big| \int_{M}s(x,m(x,y,c),c)d\eta_n(c) - \int_{M}s(x_n,m(x_n,y_n,c),c)d\eta_n(c)\big| \\
       &\leq \epsilon/2 + \int_{M}|s(x,m(x,y,c),c)d-s(x_n,m(x_n,y_n,c),c)|d\eta_n(c) \\
       & \leq \epsilon
    \end{align*}
    as needed.
\end{proof}

\subsection{General results about measurability }

\begin{lemma}\label{lem:countoonisbimeas}
    Let $p:X\rightarrow Y$ be a countable-to-one Borel map between two standard Borel spaces.
    Then $p$ is Borel bimeasurable, i.e. the image of every Borel set $B\subset X$, is Borel in $Y$.
\end{lemma}
\begin{proof}
    Fix a Borel subset $B\subset X$, and define the following set:
    $$P:=\{(x,y)\in B\times Y\ : \ p(x)=y\}$$
    Note that this is a Borel subset of $X\times Y$.
    To see that, it is enough to observe that $P$ is equal to the set $B\times Y\cap (p\times id_Y)^{-1}(\Delta(Y))$, where $p\times id_Y$ is the map from $X\times Y$ to $Y\times Y$ defined by $(x,y)\mapsto (p(x),y)$, and $\Delta(Y)\subset Y\times Y$ is the diagonal.
    
    By assumption, each section $P_y$ is countable.
    Therefore, according to \cite[Lemma 18.12]{Kechris1995}, $p(B)=\text{proj}_Y(P)$ is Borel.
\end{proof}

\begin{lemma}\label{lem:counttoonenum}
    Let $p:X\rightarrow Y$ be a countable-to-one Borel map between two standard Borel spaces.
    Then, there exists a Borel subset $A\subset \bbN\times Y$, and Borel isomorphism $\psi:X\rightarrow A$, such that the following diagram commutes:
        \begin{equation*}
             \begin{tikzcd}
                X \arrow[r, "\psi"] \arrow[d,"p"]
                & A \arrow[d, "\text{proj}_Y"] \\
                Y \arrow[r, "id_Y"] 
                & Y
            \end{tikzcd}   
        \end{equation*}
\end{lemma}
\begin{proof}
    First, consider the following equivalence relation on X:
    $$E:=\{(x,x')\in X\times X\ : \ p(x)=p(x')\}=(p\times p)^{-1}(\Delta(Y)) $$
    Where $p\times p$ is the map from $X\times X$ to $Y\times Y$ defined by $(x,x')\mapsto (p(x),p(x'))$, and $\Delta(Y)\subset Y\times Y$ is the diagonal.
    Therefore, it is clear that $E$ is a Borel subset of $X\times X$, and by assumption, all of its equivalence classes are countable.

    Fix any section $s:Y\rightarrow X$.
    It follows Lemma \ref{lem:countoonisbimeas} that $s$ is Borel.
    Indeed, choose any Borel subset $B\subset X$.
    Then, $s^{-1}(B)=p(B)$, which is a Borel.

    According to \cite[Theorem 1]{Feldman1977}, there exists a countable group $H=\{h_i\}_{i\in\bbN}$ of Borel automorphisms of $X$, such that the following holds:
    $$E = \{(x,h_i.x) \ : \ \forall x\in X, i\in \bbN\}$$
    Without loss of generality, assume that $h_1=e_H$.
    For every $x\in X$, we denote by $c_x$ the elements of $Y$ for which $x\in p^{-1}(c_x)$.
    Also, we denote by $n_x$ the minimal positive integer $i$ such that $h_i.s(c_x)=x$.
    Finaly, we define $\psi:X\rightarrow \bbN\times Y$, by $x\mapsto (n_x,c_x)$.

    We claim that $\psi$ is measurable.
    Note that it is enough to show that sets of the form $\psi^{-1}(\{i\}\times A)$ are Borel, whenever $A\subset Y$ is Borel.
    Fix such a Borel set $A\subset Y$.
    We will prove by induction on $i$ that these sets are indeed Borel.

    For $i=1$, $\psi^{-1}(\{1\}\times A)=s(A)$, and this is a Borel set, according to Lemma \ref{lem:countoonisbimeas}.
    Now, note that we have the following equality:
    $$\psi^{-1}(\{i\}\times A)=h_i.s(A)\backslash \underset{1\leq j<i}{\cup}\psi^{-1}(\{j\}\times A)$$
    By the induction hypothesis, $\psi^{-1}(\{i\}\times A)$ is Borel.

    We finish by noting that $\psi $ is injective, and therefore, it is bimeasurable by Lemma \ref{lem:countoonisbimeas}.
    Hence, $\psi$ is an isomorphism between $X$ and $\psi(X)\subset \bbN\times Y$.
\end{proof}

\begin{lemma}\label{lem:sectionsspaceseparable}
    Let $Y$ be a compact metric space, and consider the factor map $\text{proj}_Y:\bbN\times Y\rightarrow Y$.
    Fix some $\mu\in \Prob(Y)$, and denote by $\mathfrak{G}$ the space of classes of $\mu$-a.e. defined measurable sections for $\text{proj}_Y$.
    Consider the following metric $d_{\mu}$ on $\mathfrak{G}$
$$\ d_{\mu}(f,f'):=\int_Y (1-\delta_{f(c)f'(c)})d\mu(c)$$

    where $\delta_{ab}=1$ if $a=b$ and zero otherwise.
    Then, the metric space $(\mathfrak{G},d_{\mu})$ is separable.
\end{lemma}
\begin{proof}
    Fix a countable basis for the topology $\{U_i\}_i$, and denote by $\mathcal{B}$ the Boolean algebra that is generated by this basis.
    Consider the collection of sections $s:Y\rightarrow \bbN\times Y$, of the form $$s(y)=(\sum_{i=1}^m j_i \chi_{B_i}(y),y)$$ 
    where $m$ and $j_i$ run over all positive numbers, and the $B_i$'s run over all the elements of $\mathcal{B}$, and where $\chi_{B_i}$ denotes the characteristic function of $B_i$.
    Observe that this is a countable collection.
    We now show that it is dense in $(\mathfrak{G},d_{\mu})$.

    Let $f\in \mathfrak{G}$ and $\epsilon>0$.
    For every $n\in \bbN$, denote by $A_n:=f^{-1}(\{n\}\times Y)$. 
    As $Y=\cup_n A_n$, there exists $k\in \bbN$, such that $\mu(\cup_{n\geq k}A_n)<\epsilon/2$.
    Note that $\mu$ is a Radon probability measure, therefore, for every $k>n$, we can choose an open set $V_n$ such that $\mu(V_n\Delta A_n)<\epsilon/2^{k+1}$.
    For every such $V_n$, we can choose some $B_n\in \mathcal{B}$, such that $B_n\subset V_n$ and $\mu(V_n\backslash B_n)<\epsilon/2^{k+1}$.
    Finely, we define the following section $s:Y\rightarrow \bbN\times Y$:
    $$s(y)=(\sum_{n=1}^{k-1} n \chi_{B_n}(y),y)$$
    Note that we have the following inclusion:
    $$\{y\in \supp(f)\ : \ s(y)\neq f(y)\}\subset (\cup_{n<k}B_n\Delta A_n )\bigcup (\cup_{n\geq k}A_n)\subset (\cup_{n<k}(V_n\Delta A_n \cup V_n\backslash B_n))\bigcup (\cup_{n\geq k}A_n)$$
    Therefore,
    \begin{align*}
        d_{\mu}(f,s)& =\int_Y 1-\delta_{f(y)s(y)}d\mu(y)=\mu(\{y\in \supp(f)\ : \ s(y)\neq f(y)\}) \\
        & \leq \sum_{n=1}^{k-1}\mu(V_n\Delta A_n \cup V_n\backslash B_n)+\mu(\cup_{n\geq k}A_n) \\
        & \leq \sum_{n=1}^{k-1}\epsilon/2^k + \epsilon/2<\epsilon
    \end{align*}
    As needed.
\end{proof}

\subsection{The space $\Map_G(S,\Prob(X))$}\label{subsection:amenableaction}
Consider a standard Borel space $S$, endowed with a probability measure $\mu\in\Prob(S)$ and a separable Banach space $E$.
Define the following space
$$L^1(S,E):= \big\{ f:S\rightarrow E \ | \ f \text{ measurable and }\int\|f(s)\|\mathrm{d}\mu\leq \infty \big\}/\sim$$
Where $f\sim f'$ if and only if the map $s\mapsto \|f(s)-f'(s)\|$ is of $\mu$-measure zero. 
As discussed in \cite[p. 357]{Zimmer1978}, the space $L^1(S,E)$ is a Banach space. 
Its dual can be identified with the collection of all \textit{weakly measurable} maps $\lambda:S\rightarrow E^*$, such that the assignment $s\mapsto \|\lambda(s)\|$ is an element of $L^{\infty}(S)$.
See also \cite[Theorem 8.18.2]{funcanalysisandapp}.
We denote this space by $L^{\infty}(S,E^*)$.
The identification is given by $\lambda\mapsto \big( f\mapsto \Tilde{\lambda}(f)\big)$, where we define $\Tilde{\lambda}(f):=\langle \lambda , f\rangle:=\int\langle\lambda(s),f(s)\rangle\mathrm{d}\mu(s)$, for $\lambda\in L^{\infty}(S,E^*)$ and $f\in L^1(S,E)$.
Endowed with the essential sup norm, the space $L^{\infty}(S,E^*)$ becomes a Banach space, and the above identification turns into isometric isomorphism.

Suppose now that we are given a compact metric space $(X,d)$, and fix our Banach space $E$ to be $(C(X),\|\ \|_{\text{sup}})$. 
Define the following subspace of $L^{\infty}(S,C(X)^*)$:
$$\Map(S,\Prob(X)):=\big\{\lambda\in L^{\infty}(S,C(X)^*) \ | \text{ for a.e. s, } \lambda(s)\in \Prob(X) \big\}$$
Note that when endowed $\Prob(X)$ with the weak-* topology, then the space $\Map(S,\Prob(X))$ is exactly the collection of all measurable maps $\lambda:S\rightarrow \Prob(X)$.
Indeed, it is enough to notice that for such $\lambda$, the map $s\mapsto \|\lambda(s)\|$ is just the constant map $1$.
Moreover, it equals exactly to the space $B$ as in \cite[Proposition 2.2]{Zimmer1978}, when the Borel field $\{A_s\}$ is defined to be the constant field $A_s=\Prob(X)$. 
According to the proposition, this space is a compact and convex subspace of $L^{\infty}(S,C(X)^*)$.

Let $G$ be a lcsc group.
Suppose that $(S,\mu)$ is an ergodic $G$-space, and that $X$ is endowed with a continuous $G$-action.
We have a natural action of $G$ on $L^{\infty}(S,C(X)^*)$, given by $(g.\lambda)(s):=\lambda(gs)$, for $g\in G$, $\lambda\in L^{\infty}(S,C(X)^*)$ and $s\in S$.
According to the proof of \cite[Theorem 2.1]{Zimmer1978}, this is the induced adjoint action to a continuous action of $G$ on $L^1(S,C(X))$, with respect to the trivial cocycle $\alpha:S\times G\rightarrow \Iso(C(X))$, $\alpha(s,g)=Id_{C(X)}$.
Therefore, for every $g\in G$, the assignment $\lambda\mapsto g.\lambda$ is a continuous automorphism of $L^{\infty}(S,C(X)^*)$.

Clearly, the space $\Map(S,\Prob(X))$ is invariant under the action of $G$. 
Denote by $\Map_G(S,\Prob(X))$ the collection of all measurable $G$-equivariant maps $\lambda:S\rightarrow \Prob(X)$.
By all the above discussion, it follows that $\Map_G(S,\Prob(X))$ is a closed, thus compact, and convex subspace of $\Map(S,\Prob(X))$.

The motivation to introduce the space $\Map_G(S,\Prob(X))$ derived by our use of the concept of {\em amenable action}. 
See \cite{Zimmer1978}. 
A key feature of amenable action $G\acts B$, where $B$ is a Lebesgue $G$-space, is that given an affine $G$-action on a convex $weak^*$ compact set $Q\subset E^*$, where $E$ is a separable Banach space, there exists a $G$-map $\phi\in \Map_{G}(B,Q)$. In particular, if $G$ acts on a metrizable compact space $X$, then the space $\Map_{G}(B,\Prob(X))$ is not empty. 


\section{Isometric ergodicity}\label{sec:ergodic}

Let $G$ be a lcsc group. A Lebesgue $G$-space $(X,\nu)$ is {\em isometrically ergodic}, if for every separable metric space $Y$ on which $G$ acts by isometries, and every $G$-equivariant map $\varphi:X\rightarrow Y$, $\varphi$ is essentially constant. 

Given a Borel map $q:\mathcal{M}\rightarrow V$ between standard Borel spaces, a metric on q is a Borel function $d:\mathcal{M}\times_V\mathcal{M}\rightarrow [0,\infty]$ whose restriction $d_v$ to each fiber $\mathcal{M}_v=q^{-1}(\{v\})$ is a separable metric. A {\em fiber-wise isometric $G$-action} on such $\mathcal{M}$ consists of q-compatible actions $G\acts\mathcal{M},G\acts V$, so that the maps between the fibers $g:\mathcal{M}_v\rightarrow \mathcal{M}_{gv}$ are isometries.

Suppose we have actions of $G$ on two such Borel spaces $\mathcal{M}$ and $V$, and a Borel map $q:\mathcal{M}\rightarrow V$ in which all the fibers are countable. 
Define $d:\mathcal{M}\times_V \mathcal{M}\rightarrow [0,\infty]$ to be the trivial metric, that is, $d(m,n)=\delta_{mn}$. 
Then, the action of $G$ on $\mathcal{M}$, equipped with the metric $d$ above on q, is an example of a fiber-wise isometric $G$-action.

A map $p:A\rightarrow B$ between Lebesgue $G$-spaces, is {\em relatively isometrically ergodic} if for every fiber-wise isometric action on $q:\mathcal{M}\rightarrow V$, and $q$-compatible maps $f:A\rightarrow \mathcal{M}$, $f_0:B \rightarrow V$, there exists a compatible $G$-map $f_1:B\rightarrow \mathcal{M}$ making the following diagram commutative:

\begin{equation}\label{equ:relativeergod}
 \begin{tikzcd}
A \arrow[r, "f"] \arrow[d, "p"]
& \mathcal{M} \arrow[d, "q"] \\
B \arrow[r, "f_0"] \arrow[ru, dashrightarrow, "f_1"]
& V
\end{tikzcd}   
\end{equation}

We say that a pair of Lebesgue $G$-spaces $(B_-,B_+)$ is a {\em boundary pair} if the actions of $G$ on $B_+$ and $B_-$ are amenable, and the projection maps $p_{\pm}:B_-\times B_+\rightarrow B_{\pm}$ are relatively isometrically ergodic.

\begin{remark}
Let $(B_-,B_+)$ be a boundary pair for $G$.
\begin{enumerate}
    \item For a Lebesgue $G$-spaces A and B, if the projection $A\times B\rightarrow B$ is relatively isometrically ergodic, then, A is isometrically ergoic. See \cite[Proposition 2.2.(iii)]{Bader2014}.

    In particular, both $B_+$ and $B_-$ are isometrically ergodic.
    \item The action of $G$ on $B_-\times B_+$ is ergodic. 
    Indeed, if $S\subset B_-\times B_+$ is a Borel $G$-invariant subset, then its characteristic map $\chi_S$ is a $G$-invariant map into the set $\{0,1\}$.
    Completing diagram \ref{equ:relativeergod} with $f=\chi_A$, $A=B_-\times B_+$, $B=B_+$ and $V=\{pt\}$, will generate a Borel $G$-invariant map $f_1:B_+\rightarrow \{0,1\}$.
    It follows by the previous remark, that $f_1$ must be constant.
\end{enumerate}

\end{remark}

\begin{theorem}[{\cite[Theorem 2.7]{Bader2014}}]\label{thm:mainboundarypair}
Let $\mu_+$ be a probability measure on $G$, that is continuous with respect to the Haar measure and is not supported on a proper closed sub-semigroup.
Define the probability measure $\mu_-$ to be the inverse of $\mu_+$. That is, $\mu_-(A)=\mu_+(A^{-1})$.
We denote by $(B_+,\nu_+)$ (resp. by $(B_-,\nu_-)$) the Furstenberg-Poisson boundary corresponds to $(G,\mu_+)$ (resp. to $(G,\mu_-)$).
See \cite[section 2]{Bader2006} and references therein for a comprehensive discussion on the Furstenberg-Poisson boundary.
The pair $(B_-,B_+)$ is a boundary pair for $G$ and for any of its lattices.
\end{theorem}

We now turn to apply the theory of boundary pairs in the case of convex minimal actions on sclocc median algebras. 
For the definition of the space $\Jc_n^{\subseteq}$, see \ref{def:joinspace}

\begin{proposition}\label{prop:match}
Let $G$ be a countable discrete group, $M$ a sclocc $G$-median algebra and suppose that $G\acts M$ is convex minimal. 
Let $(B_-,B_+)$ be a boundary pair for $G$, and suppose that we have two Borel $G$-maps $\varphi_{\pm}:B_{\pm}\rightarrow \Jc_n$. Then, for almost every $(\theta_-,\theta_+)$, $$\pi_{\varphi_+(\theta_+)}(\varphi_-(\theta_-))=\varphi_+(\theta_+)$$  
\end{proposition}
\begin{proof}
Consider the map $\mathfrak{f} :B_-\times B_+ \rightarrow \Jc_n^{\subseteq}$ defined by $$(\theta_-,\theta_+)\mapsto (\varphi_+(\theta_+),\pi_{\varphi_+(\theta_+)}(\varphi_-(\theta_-))$$ 
It is a measurable $G$-map. For the measurably part, see lemma \ref{lem:continousmapping}.   
As $(B_-,B_+)$ is a boundary pair for $G$, there is a compatible $G$-map $\varphi': B_+\rightarrow  \Jc_n^{\subseteq}$ making the following diagram commutative:

\begin{center}
    \begin{tikzpicture}
  \matrix (m) [matrix of math nodes,row sep=3em,column sep=4em,minimum width=2em]
  {
     B_-\times B_+ &  \Jc_n^{\subseteq} \\
     B_+ & \Jc_n\\};
  \path[-stealth]
    (m-1-1) edge node [left] {$p_+$} (m-2-1)
            edge node [above] {$\mathfrak{f}$} (m-1-2)
    (m-2-1) (m-2-1) edge node [above] {$\varphi_+$} (m-2-2)
            edge [dashed] node [above] {$\varphi'$} (m-1-2)
    (m-1-2) edge node [left] {$p$} (m-2-2);
\end{tikzpicture}
\end{center}

where $p_+:B_-\times B_+\rightarrow B_+$ and $p: \Jc_n^{\subseteq}\rightarrow \Jc_n$ are the natural propjections.
Note that it follows by the diagram, that for almost every $ (\theta_+,\theta_-)$ $$(\dagger)\fs(\theta_+,\theta_-)=\varphi'(\theta_+)= (\varphi_+(\theta_+),\pi_{\varphi_+(\theta_+)}(\varphi_-(\theta_-)))$$

For the sake of contradiction, suppose there exists a subset $A\subset B_-\times B_+$ of positive $\nu_-\otimes\nu_+$-measure, such that for every $(\theta_+,\theta_-)\in A$, $$(\ddagger)\pi_{\varphi_+(\theta_+)}(\varphi_-(\theta_-))\varsubsetneq\varphi_+(\theta_+)$$

As $A$ is $G$-invariant, it follows by ergodicity that it is of full measure. By Fubin's theorem, we can find $\theta_+\in B_+$ and $\Omega\subset B_-$ with full $\nu_-$-measure, such that $(\dagger)$ and $(\ddagger)$ holds for the pair $(\theta_+,\theta_-)$, for every $\theta_-\in \Omega$. In particular, for every $\theta,\theta'\in \Omega$
$$I':=\pi_{\varphi_+(\theta_+)}(\varphi_-(\theta))=\pi_{\varphi_+(\theta_+)}(\varphi_-(\theta'))\varsubsetneq \varphi_+(\theta_+)=:I$$

Fix $x\in I\backslash I'$, and a closed half-space $\hs\in\Delta(I',x)$. It follows from lemma \ref{lem:fio2}.(1), that for $\theta\in \Omega$, $\varphi_-(\theta)\subset \hs$. Since we may assume $\Omega$ is $G$-invariant, we got a contradiction to the minimality of the action: the closure of the convex hull of the union of all $\varphi_-(\theta)$ for $\theta\in \Omega$, is a $G$-invariant closed convex set that not contains the element x. Therefore, the proposition holds. 

\end{proof}



    



\section{Proof of Theorems \ref{thm:stationary-nofactors}, \ref{thm:nocubicalfactor} and \ref{thm:main}}
\label{sec:theoremnocubicalfactor}

In what follows, the first two subsections are devoted to the proof of theorem \ref{thm:nocubicalfactor} and the latter subsections are devoted to the proofs of 
theorems \ref{thm:stationary-nofactors} and \ref{thm:main} correspondingly. 

\subsection{Proof of Theorem \ref{thm:nocubicalfactor}: Existence of boundary maps}
Let M be a minimal sclocc $G$-median algebra, with countable intervals and with no cubical factors. 
Let $(B_-,\nu_-)$ and $(B_+,\nu_+)$ be two $G$-Lebesgue spaces, such that the pair $(B_-,B_+)$ is a boundary pair for $G$.
In this subsection we will prove the existence of a $G$-map $B_-\to M$.
The proof for $B_+$ is done in a similar way.

Observe that the self-median operator $\Phi$ acts continuously on the compact convex spaces 
\[ \Map_{G}(B_{\pm},\Prob(M)). \] 
These spaces are not empty, as the actions of $G$ on $B_+$ and $B_-$ are amenable. 
See subsection \ref{subsection:amenableaction}.
It follows by Tychonoff's theorem that $\Map_{G}(B_{\pm},\Prob(M))^{\Phi}=\Map_{G}(B_{\pm},\Prob(M)^{\Phi})\neq \varnothing$. See \cite{tychonoffixedpoint}. 
For the rest of this section we fix two maps $$\varphi_{\pm}\in \Map_{G}(B_{\pm},\Prob(M)^{\Phi}).$$ In sake of abbreviation, we might confuse between an element $\theta\in B_{\pm}$ and its image under the appropriate map.

Fix a countable collection of closed half-spaces, as in corollary \ref{cor:separatingcollection}, taking any additional half-spaces so that the collection will be $G$-invariant, and denote it by $\Hc$. 
For every $\theta_{\pm} \in B_{\pm}$ and every $i\in \{0,1/2,1\} $, we define the following subset of $\Hc$: $$\Hc_{\theta}^i:=\{\hs \ | \ \theta_{\pm}(\hs)=i\}$$
As almost every $\theta_{\pm}$ is balanced, it follows by theorem \ref{thm:balancediscubical} that its support constitutes a subcube of $M$. 
Therefore, the sets $\Hc_{\theta_{\pm}}^0, \Hc_{\theta_{\pm}}^{1/2}$ and $\Hc_{\theta_{\pm}}^1$ constitute a partition of $\Hc$.

Among all these three sets, only those of the form $\Hc_{\theta_{\pm}}^{1/2}$ will be in use. 
So let us denote them with the abbreviated notation $\Hc_{\theta\pm}$.
Consider the following map:
\begin{align*}
  (\dagger) \   & \psi: B_-\times B_+\rightarrow (\bbN\cup \infty)  \\
    & (\theta_-,\theta_+)\mapsto |\Hc_{\theta_-}\Delta\Hc_{\theta_+}|
\end{align*}

According to corollary  \ref{cor:measurabilitysubshalfspa}, this map is Borel and clearly it is $G$-invariant. As $G$ acts ergodicly on $B_-\times B_+$, the essential image of $\psi$ is some constant $m\in \bbN\cup \{\infty\}$. 
We claim that the essential image is in fact zero.
This will promise us that the image of $\varphi_+$ inside $\Prob(M)$ consists of delta measures. 

We would like to draw the reader's attention, that the use of the map $(\dagger)$ is greatly inspired by the proof of theorem 7.1 in \cite{fer18}, although the analysis is substantially different.

For a generic $\theta_{\pm}\in B_{\pm}$, we set $C_{\pm}:=C_{\theta_{\pm}}:=\supp\big(\varphi_{\pm}(\theta_{\pm})\big)$, $I_{\pm}:=I_{\theta_{\pm}}:=\overline{\Conv(C_{\pm})}=\Conv(C_{\pm})$, and we denote by $\pi_{\pm}$ or $\pi_{\theta_{\pm}}$ the gate-projections of $I_{\pm}$. 
As the convex hull of a cube is an interval, the assignments $\theta_{\pm}\mapsto I_{\pm}$ are maps from $B_{\pm}$ to $\Jc_2$.
We consider also the assignments $I\mapsto \supp(I)$ from $\Jc_2$ to $\Cub(M)$, and $C$ mapped to its unique fully supported balanced measure $\lambda_C$, from $\Cub(M)$ to $\Prob(M)^{\Phi}$. 
It follows by lemmata \ref{lem:cubetoint}, \ref{lem:cubetobalanced} and corollary \ref{cor:inttocubeismeasurable}, that all these maps and their various compositions, are measurable. 
It is easy to check that they are also $G$-maps.




Let us show that $m=0$.
For the sake of contradiction suppose that for almost every $(\theta_-,\theta_+)$, $$\Hc_{\theta_+}\Delta \Hc_{\theta_-}\neq \varnothing$$
By proposition \ref{prop:match}, for almost every $(\theta_-,\theta_+)$, $I_+=\pi_+(I_-)$ and $I_-=\pi_-(I_+)$. 
The former assertion implies that $\Hc_{\theta_+}\subset \Hc_{\theta_-}$, while the latter means that $\Hc_{\theta_-}\subset \Hc_{\theta_+}$.
This is a contradiction, thus, we may assume from now on that $m=0$.

Let $\Omega\subset B_-$ be a Borel set of $\nu_-$-measure 1, and $\theta_+\in B_+$ such that for every $\theta_-\in \Omega$ 
$$\Hc_{\theta_+}=\Hc_{\theta_-}$$
In particular, for every $\theta,\theta'\in \Omega$ $$\Hc':=\Hc_{\theta}=\Hc_{\theta'}$$
Moreover, we hvae
$$\Hc(I_{\theta})=\Hc_{\theta}=\Hc_{\theta'}=\Hc(I_{\theta'})$$
Therefore, it follows by lemma \ref{lem:convgateisomorphism} that $I_{\theta}$ and $I_{\theta'}$ are isomorphic, and the isomorphisms are given by $\pi_{\theta}$ and $\pi_{\theta'}$. 

Let us now show that all these intervals are in fact points.
Fix $\theta_0\in \Omega$. 
By lemma \ref{lem:InducedAction} and the above discussion, we may consider the induced action of $G$ on $I_0:=I_{\theta_0}$. 
Observe that this action is minimal.
Indeed, as $\pi_0:=\pi_{I_0}$ is a continuous median $G$-map, the pre-image of any closed $G$-invariant median subalgebra of $I_0$, is a median subalgebra with the same properties of M. 

In particular, since $C_0:=C_{\theta_0}$ is a $G$-invariant median subalgera of $I_0$, we get that $I_0=C_0$. 
As $\theta_0$ is generic, it follows that for almost every $\theta\in B_-$, $I_{\theta}=C_{\theta}$.

But now, $\pi_0$ is a surjective morphism onto a cube.
Since $M$ has no cubical factors, $C_0$ must be a point.
Similarly, $C_{\theta}$ is a point, for every $\theta \in \Omega$.
That is, $\theta\mapsto C_{\theta}$ can be viewed as a Borel $G$-map between $B_-$ to $M$.
For $\theta\in B_-$, we denote its corresponding point in $M$ by $x_{\theta}$.


\subsection{Proof of Theorem \ref{thm:nocubicalfactor}: Uniqueness of boundary maps}
We now show that there are no other $G $-maps from $B_-$ into $\Prob(M)$.
Denote by $\phi_{\pm}:B_{\pm}\rightarrow M$ our boundary maps, and suppose we have another $G$-map $\psi:B_-\rightarrow \Prob(M)$ .

Define $Y=M\times \Prob(M)$ and set $T:=\{(x,y,z)\in M^3 \ : \ y\in [x,z]\}$.
We denote by $p_i:T\rightarrow M$ the projection onto the $i'th$ coordinate of elements in $T$.
Define the following space:
$$X:=\{\eta\in\Prob(T)\ | \ \exists a_{\eta}\in M \text{ s.t. }(p_1)_*(\eta)=\delta_{a_{\eta}}, \text{ and } p_3 \text{ is a measure iso.}\}$$

We have a natural projection $q:X\rightarrow Y$, namely $\eta\mapsto (a_{\eta},(p_3)_*(\eta))$.
This map will fit to the right side of diagram \ref{equ:relativeergod}.
For some purposes, it will be convenient to give an equivalent description of $X$.
For every $(a_y,\eta_y)=y\in Y$, define the following space:
$$X_y:=\{f\in \Map ( (M,\eta_y), M)\ | \eta_y- \text{a.e. } c\in M, (a_y,f(c),c)\in T\} $$

Here, the space $\Map ( (M,\eta_y), M)$ denotes the space of classes of measurble functions defined up to $\eta_y$-almost everywhere.
We claim that $X_y$ is isomorphic to the fiber of $y$ in $X$ under the projection $q$.
Indeed, fix $f\in X_y$, and define the following positive functional on $C(T)$:
$$\forall s\in C(T):s\mapsto\int_M s(a_y,f(c),c)d\eta_y(c) $$

Denote by $\eta_{y,f}$ the corresponding measure, and let us check that $\eta_{y,f}\in q^{-1}(y)$.
It is rather easy to verify that $(p_1)_*(\eta_{y,f})=\delta_{a_y}$ and that $(p_3)_*(\eta_{y,f})=\eta_y$.
Lastly we note that the $\eta_y$-a.e. defined map $\overline{f}:M\rightarrow T$, $c\mapsto (a_y,f(c),c)$, is the measurable inverse of $p_3$ as a map with the domain $(T,\eta_{y,f})$.
Given a measure $\eta\in q^{-1}(y)$, denote by $p_3^{-1}$ the measurable inverse of $p_3:(T,\eta)\rightarrow (M,\eta_y)$, then indeed, $\eta=\eta_{y,p_2\circ p_3^{-1}}$. 

We now move to define a metric on the projection $q$, in the sense of section \ref{sec:ergodic}. 
Given two elements $a,b\in M$, denote by $\delta_{ab}$ the function that equals 1 if $a=b$, and 0 otherwise.
For $(a_y,\eta_y)=y\in Y$ and two functions $f,f'\in X_y$, define $d_y(f,f')$ to be the following value:
$$(*) \ d_{\eta_y}(f,f'):= d_y(f,f'):=\int_M 1-\delta_{f(c)f'(c)}d\eta_y(c)$$
It is rather simple to check that this is indeed a metric. 
We claim that the space $(X_y,d_y)$, is separable metric space. 
In addition, as described in section \ref{sec:ergodic}, we need also to check that the map we get $d:X\times_Y X\rightarrow [0,\infty]$ is Borel.
See Lemmata \ref{lem:fibersseparable} and \ref{lem:distisborel} for the details.

The diagonal action of $G$ on $M^3$, induces a natural action on $T$.
As this action commutes with both $p_1$ and $p_3$, it gives rise to an induced action on $X$.
Finally, consider the natural action of $G$ on $Y$. 
Let us define an action on $X$, which turns $q$ into a $\G$-map:
Given $\eta\in X$ and $g\in G$, the measurable inverse of $p_3$ with respect to $g_*\eta$ is $g_*(g \cdot p_3^{-1})$, where $p_3^{-1}$ is the measurable inverse with respect to $\eta$. 
Therefore, the action of $g\in G$ on the fiber $X_y$ for $y\in Y$ translates to
$$f\in X_y\mapsto g_*(g \cdot f)\in X_{g y} $$
The above turns the action of $G$ on $q$ into a fiber-wise isometric action. Indeed, if $y\in Y$, and $f,f'\in X_y$:
\begin{align*}
    d_{g y}(g_*(g \cdot f),g_*(g \cdot f')) & = \int_M 1-\delta_{g_*(g \cdot f)(c)g_*(g \cdot f')(c)}dg_*\eta_y(c) \\
    & = \int_M 1-\delta_{g_*( f)(c)g_*(f')(c)}dg_*\eta_y(c) \\
    & = \int_M 1-\delta_{g_*( f)(g c) g_*(f')(g c)}d\eta_y(c) \\
    & = \int_M 1-\delta_{f(c)f'(c)}d\eta_y(c) \\
    & = d_y(f,f')
\end{align*}

As needed.

Let us now define the horizontal arrows in diagram \ref{equ:relativeergod}.
The lower arrow, namely $\varphi:B_-\rightarrow Y$, is defined by $\theta_-\mapsto (\phi_-(\theta_-),\psi(\theta_-))$.
We define the upper horizontal arrow in terms of the fibers of $q$.
Given a generic pair $(\theta_-,\theta_+)\in B_-\times B_-$, we define a measurable map $f_{(\theta_-,\theta_+)}:M\rightarrow M$ by $f_{(\theta_-,\theta_+)}(c)=m(\phi_-(\theta_-),\phi_+(\theta_+),c)$.
Finally, the map $\varphi':B_-\times B_+\rightarrow X $ is defined as follows:
$$\varphi'(\theta_-,\theta_+)=f_{(\theta_-,\theta_+)}\in X_{(\phi_-(\theta_-),\varphi(\theta_-))}$$
We claim that $\varphi'$ is measurable. 
Consider the map from $B_-\times B_+\rightarrow M^2\times \Prob(M)$, defined by $$(\theta_-,\theta_+)\mapsto (\phi_-(\theta_-),\phi_+(\theta_+),\varphi(\theta_-))$$
Clearly it is measurable, and note that $\varphi'$ is just the composition of this map with the map $\mu_{(.)}$ in \ref{lem:boundarypairintoxmeasurable}, and therefore, it is also measurable.
It is easy to verify that it's also a $G$-map.
Indeed, given $(\theta_-,\theta_+)\in B_-\times B_+$, $g\in G$ and $c\in M$
\begin{align*}
    f_{(g\theta_-,g\theta_+)}(c) & = m(g\phi_-(\theta_-),g\phi_+(\theta_+),gg^{-1}c)=g m(\phi_-(\theta_-),\phi_+(\theta_+),g^{-1}c) \\
    & = \big(g_*(g\cdot f_{(g\theta_-,g\theta_+)})\big)(c)
\end{align*}
and note that the last line is precisely the action of $g$ on the elements of $X_{\varphi(\theta_-)}$.

The relatively isometrically ergodic projection $B_-\times B_+\rightarrow B_-$ gives us a map $\varphi'' :B_-\rightarrow X$ that completes the diagram:

\begin{equation*}
 \begin{tikzcd}
B_-\times B_+ \arrow[r, "\varphi'"] \arrow[d]
& X \arrow[d, "q"] \\
B_- \arrow[r, "\varphi"] \arrow[ru, dashrightarrow, "\varphi''"]
& Y
\end{tikzcd}   
\end{equation*}

We claim that for almost every $(\theta_-,\theta_+)\in B_-\times B_+$, and $\psi(\theta_-)$-a.e. $c\in M$ $(\ddagger) \varphi'(\theta_-,\theta_+)(c)=\varphi''(\theta_-)(c)=\phi_-(\theta_-)$.
Suppose this is not the case and note that that the collection of points for which $(\ddagger)$ doesn't hold, defines a measurable set in $B_-\times B_+$. 
For details, see \ref{lem:measurablecondition}.

This collection is invariant under the action of $G$, and therefore, by ergodicity, for a.e. $(\theta_-,\theta_+)\in B_-\times B_+$, there exists $c\in \supp(\psi(\theta_-))$, with $\varphi'(\theta_-,\theta_+)(c)\neq \phi_-(\theta_-)$.
Fix $\Omega\subset B_+$ of measure 1, and $\theta_-\in B_-$, such that for every $\theta_+\in \Omega$,  $\varphi'(\theta_-,\theta_+)=\varphi''(\theta_-)$.
By the above, there exists $c\in \supp(\psi(\theta_-))$, such that for every $\theta_+\in \Omega $, $x=\varphi''(\theta_-)(c)=\varphi'(\theta_-,\theta_+)(c)\neq \phi_-(\theta_-)$.
This is in turn implies a contradiction, as $M$ is a minimal $G$- median algebra.
Indeed, take a closed half-space $\hs\in \Delta(x,\phi_-(\theta_-))$.
By the construction of $\varphi'$ and convexity, it follows that $\phi_+(\theta_+)\in \hs$, for every $\theta_+\in \Omega$.
But now, the closure of the convex hull of all those $\phi_+(\theta_+)$ is contained in $\hs\varsubsetneq M$, and thus, it constitutes a proper $G$-invariant median subalgebra.
Contradiction!

We finish by showing that for a.e. $\theta_-\in B_-$, $\psi(\theta_-)=\delta_{\phi_-(\theta_-)}$.
Suppose this is not the case. 
Consider the $G$-map from $B_-$ to $(\mathcal{C}M)^2$, defined by $\theta\mapsto (\supp(\psi(\theta)),\phi_-(\theta))$.
By assumption, the preimage of the diagonal $\Delta \mathcal{C}M\subset (\mathcal{C}M)^2$ is a $G$ invariant measurable set that is not of full measure.
Therefore, by ergodicity, again, for a.e. $\theta_-$, $\psi(\theta_-)\neq\delta_{\phi_-(\theta_-)}$.

Fix $\Omega\subset B_+$ and $\theta_-\in B_-$ as above, $c\in \supp(\psi(\theta_-))$ different than $\phi_-(\theta_-)$, and a closed half-space $\hs\in\Delta(\phi_-(\theta_-),c)$.
It follows the discussion above that now, $\phi_+(\theta_+)\in \hs$ for every $\theta_+\in \Omega$.
This implies a contradiction, as required.

\begin{lemma}\label{lem:fibersseparable}
    The space $(X_y,d_y)$ is a separable metric space.
\end{lemma}
\begin{proof}
    Consider the space $Z:= \{ (b,c)\in M^2\ : \ b\in [a_y,c]\}$ and the natural projection $p:Z\rightarrow M$, $(b,c)\mapsto c$.
    Note that $Z$ is a compact metric space, as it is the image of the continuous map $k:M\times M\rightarrow M\times M$, that defined by $(z,c)\mapsto (m(a_y,z,c),c)$.
    Hence, we are in the situation of Lemma \ref{lem:counttoonenum}.
    Let $\psi:Z\rightarrow \bbN\times M$ be the bimeasurable, injective map such that the following diagram commute:
    \begin{equation*}
         \begin{tikzcd}
            Z \arrow[r, "\psi"] \arrow[d,"p"]
            & \bbN\times M \arrow[d, "\text{proj}_M"] \\
            M \arrow[r, "id_M"] 
            & M
        \end{tikzcd}   
    \end{equation*}
    Denote by $\mathfrak{G}$ the space of classes of $\eta_y$-a.e. defined measurable sections for $\text{proj}_M$.
    We have a natural injective map $\mathfrak{f}:X_y\rightarrow \mathfrak{G}$, defined by $f\mapsto \psi\circ f$.
    Observe that as $\psi$ injective, for every $f,f'\in X_y$ we have $$\{c\in \supp(f)\cap \supp(f')\ : \ f(c)\neq f'(c)\}=\{c\in \supp(f)\cap \supp(f')\ : \ \psi(f(c))\neq \psi(f'(c))\}$$
    That is, $\mathfrak{f}$ is an isometry from $(X_y,d_y)$ into $(\mathfrak{G},d_{\eta_y})$.
    By lemma \ref{lem:sectionsspaceseparable}, the latter is separable metric space, hence, so is the former.
\end{proof}

\begin{lemma}\label{lem:distisborel}
    The map $d:X\times_Y X \rightarrow [0,1]$ is Borel.
\end{lemma}
\begin{proof}
    Here the Borel structure on $X\times_Y X$ is the one induced by the ambiant space $X\times X$, where $X\subset \Prob(T)$ is eqquiped with the weak* topology.

    So first we want to describe the distance $d_y$ in terms of elements in $X$.
    Let $x_1,x_2\in X$, such that $q(x_1)=q(x_2)=y\in Y$, and let $l^i:M\rightarrow T$ be the corresponding inverses for $p_3$.
    Then,
    \begin{align*}
        d_y(x_1,x_2)& =d_y(p_2\circ l^1, p_2\circ l^2) = \eta_y(\{c\in M\ : \ p_2(l^1(c))\neq p_2(l^2(c))\}) \\
        & = \eta_y(\{c\in M\ : \ l^1(c)\neq l^2(c)\}) = \eta_y(p_3(\supp(x_1)\Delta \supp(x_2))) \\ 
        & = x_1(\supp(x_1)\backslash \supp(x_2)) + x_2(\supp(x_2)\backslash \supp(x_1)) \\
        & = 2-x_1( \supp(x_2))-x_2( \supp(x_1))
    \end{align*}
    So let us show that the map $(\mu,\eta)\mapsto \mu(\supp(\eta))$ is a measurable map from $\Prob(T)^2$ to [0,1].

    First, let $\mathcal{C}T$ denote the space of closed subsets of $T$, equipped with the Chabauty topology.
    We claim that the map $v:\Prob(T)\mapsto \mathcal{C}T$, defined by $\mu\mapsto \supp(\mu)$, is measurable.
    Given an open set $U\subset T$, the map $t_U:\Prob(T)\rightarrow [0,1]$ that evaluates every measure $\mu$ at $U$, is a measurable map (in fact, it is a lower semi-continuous map).
    Therefore, the following sets are evidentially measurable:
    \begin{align*}
        & v^{-1}\big(\{F \in \mathcal{C}T\ : \ F\cap U\neq \varnothing\} \big)=t_U^{-1}((0,1]) \\
        & v^{-1}\big(\{F \in \mathcal{C}T\ : \ F\subset U\}\big)  =\bigcup_n t_{U_n}^{-1}(\{1\})
    \end{align*}
    Where $U_n=T\backslash \overline{V_{1/n}(U^c)}$, and where $V_{\epsilon}(F)$ is the epsilon neighborhood of $F$ in $T$.
    This implies the measurability of $v$.

    Finely, let us show that the map $u:\Prob(T)\times \mathcal{C}T\rightarrow [0,1]$, given by $(\mu,F)\mapsto \mu(F)$ is measurable, and more specifically, upper semi-continuous.
    Suppose that $(\mu_n,F_n),(\mu,F)\in \Prob(T)\times \mathcal{C}T$ are such that $(\mu_n,F_n)\xrightarrow{n}(\mu,F)$.
    For every $m\in\bbN$, we know that $\underset{n\rightarrow \infty}{\lim} \sup \mu_n(\overline{V_{1/m}(F)})\leq \mu(\overline{V_{1/m}(F)})$.
    In addition, for $n\gg 0$ $F_n\subset V_{1/m}(F)\subset \overline{V_{1/m}(F)}$. 
    Therefore, for every $m\in\bbN$, $\underset{n\rightarrow \infty}{\lim}\sup \mu_n(F_n)\leq \mu(\overline{V_{1/m}(F)})$.
    We end by noting that $\underset{m\rightarrow \infty}{\lim}\mu(\overline{V_{1/m}(F)})=\mu(F)$.
    
\end{proof}

\begin{lemma}\label{lem:measurablecondition}
    In the above notations, the condition $(\ddagger)$ defines a measurable set in $B_-\times B_+$.
\end{lemma}
\begin{proof}
    The relation described in $(\ddagger)$, defines the collection of pairs $(\theta_-,\theta_+)\in B_-\times B_+$ such that $\supp \varphi'(\theta_-,\theta_+)= \{\phi_-(\theta_-)\}^2\times \supp\psi(\theta_-)$.
    Consider the map $\supp\circ \varphi':B_-\times B_+\rightarrow \mathcal{C}T\subset\mathcal{C}M^3$, which is measurable, and note that the collection above is precisely the preimage of the set $\Delta M\times \mathcal{C}M\subset \mathcal{C}(M^3)$.
\end{proof}

\subsection{Proof of Theorem \ref{thm:stationary-nofactors}}
Let $G$, $M$ and $\mu$ be as in the theorem and set $\mu_+=\mu$ and $\mu_-=\iota_*(\mu)$, for $\iota(g)=g^{-1}$.
Let $(B_+,\nu_+)$ be the Furstenberg-Poisson boundary corresponds to $(G,\mu_+)$, and $(B_-,\nu_-)$ the Furstenberg-Poisson boundary corresponds to $(G,\mu_-)$.
According to Theorem \ref{thm:mainboundarypair}, the pair $(B_-,B_+)$ is a boundary pair.
Thus, by Theorem \ref{thm:nocubicalfactor},
there exists a unique $G$-map from $B_+$ to $\Prob(M)$.
The theorem now follows by Furstenberg's correspondence, see \cite[Theorem 2.16]{Bader2006}.

\subsection{Proof of Theorem \ref{thm:main}}\label{sec:theoremgeneral}
Let $M$, $G$, $(B_-,\nu_-)$ and $(B_+,\nu_+)$ be as in Theorem \ref{thm:main}, and $C$ and $M'$ as promised by Proposition \ref{prop:structure}.
As in the previous section, we will show the statement for $B_-$.
Therefore, the space $B_+$ is mentioned here merely for the application of theorem \ref{thm:nocubicalfactor}.

Denote by $p_{M'}$ the natural projection from $M$ onto $M'$.
This is a $G$-equivariant continuous median map.
Thus, the action of $G$ on $M'$ must be minimal as well.

As $M'$ has no cubical factor, it follows by theorem \ref{thm:nocubicalfactor} that we have a unique $G$-map $\theta\mapsto x_{\theta}$, from $B_-$ to $M'$, which is also the unique map from $B_-$ to $\Cub(M')$.
This map gives rise to a map $\varphi: B_-\rightarrow \Cub(M)$, defined by $\theta\mapsto \{x_{\theta}\}\times C$.


Suppose we have another $G$-map $\varphi':B_-\rightarrow \Cub(M)$.
For a generic $\theta\in B_-$, denote by $I_{\theta}'$ the corresponding convex hull of $\varphi'(\theta)$.
Observe that $p_{M'}$ maps intervals to intervals.
Therefore, we have the following $G$-map:
$$B_-\rightarrow\Cub(M'),\ \theta\mapsto \Ends(p_{M'}(I_{\theta}'))$$

By the uniqueness of $x_{\theta}\in M'$, it follows that $p_{M'}(I_{\theta}')=\{x_{\theta}\}$ for a.e. $\theta\in B_-$.
That is, for a.e. $\theta\in B_-$, $\varphi'(\theta)=\{x_{\theta}\}\times C_{\theta}''$, for some subcube $C_{\theta}''\subset C$.

That is, we have a $G$-equivariant assignment from $B_-$ to the finite set $\Cub(C)$, namely, $\theta\mapsto C_{\theta}''$.
By metric ergodicity, this map is essentially constant.
Denote by $C'\subset C$ its essential image.
As in the case of $M'$, the action of $G$ on $C$ must be minimal as well, and since $C'$ is $G$-invariant, we get an equality $C'=C$.
In other words, $\varphi'$ is nothing but the map $\varphi$.

\bibliographystyle{alpha}
\bibliography{endtobig.bib}
\end{document}